\documentclass[12pt,a4paper]{article}
\usepackage{graphicx,psfrag,amsmath,amsfonts,verbatim}
\usepackage{amssymb}
\usepackage{mathrsfs}
\usepackage{amsthm}
\usepackage{enumerate}
\usepackage{geometry}
\usepackage{dsfont}
\usepackage{bm}
\usepackage{cite}
\usepackage{appendix}
\usepackage{url}
\usepackage{color}
\usepackage{multirow}
\usepackage{threeparttable}
\usepackage{booktabs}
\usepackage{array}
\usepackage{subcaption}
\usepackage{accents}
\usepackage{statex}
\usepackage{cases}
\usepackage{makecell}
\usepackage[colorlinks=true,linkcolor=black,anchorcolor=black,citecolor=black,urlcolor=blue]{hyperref}
\makeatletter

\def\wideubar{\underaccent{{\cc@style\underline{\mskip8mu}}}}

\newtheorem{theorem}{Theorem}[section]
\newtheorem{proposition}{Proposition}[section]

\title{Globalized distributionally robust chance-constrained support vector machine based on core sets}
\author{Yueyao Li, Chenglong Bao, Wenxun Xing}
\date{}

\begin{document}

\maketitle

\begin{abstract}
Support vector machine (SVM) is a well known binary linear classification model in supervised learning. This paper proposes a globalized distributionally robust chance-constrained (GDRC) SVM model based on core sets to address uncertainties in the dataset and provide a robust classifier. The globalization means that we focus on the uncertainty in the sample population rather than the small perturbations around each sample point. The uncertainty is mainly specified by the confidence region of the first- and second-order moments. The core sets are constructed to capture some small regions near the potential classification hyperplane, which helps improve the classification quality via the expected distance constraint of the random vector to core sets. We obtain the equivalent semi-definite programming reformulation of the GDRC SVM model under some appropriate assumptions. To deal with the large-scale problem, an approximation approach based on principal component analysis is applied to the GDRC SVM. The numerical experiments are presented to illustrate the effectiveness and advantage of our model.
\end{abstract}

\section{Introduction}\label{sec:intro}
Support vector machine (SVM) is a maximum-margin classifier proposed by Cortes and Vapnik\cite{drc:svn}. As a well known supervised learning algorithm, SVM has been extensively studied and widely applied in biomedical science\cite{svm:bio}, image recognition\cite{svm:image}, energy forecasting\cite{svm:energy} and many other fileds.

Classical SVM models are constructed under the assumption that the values of the sample points are known exactly. However, in the real world, the data points may involve uncertainties caused by inaccurate measurements, sampling and modelling errors. For example, in image classification applications, the image noise leads to features uncertainty\cite{svm:Bi}; in case of gene/protein expression data, uncertainties arise due to biological heterogeneity\cite{svm:drcc:Ben-Tal}. Distribution shift can also be regarded as a kind of uncertainty of data distributions, in which training and testing data come from different distributions. It may result in poor performance of the classifier if ignoring the uncertainty in the data. 

Robust optimization, as a useful method handling data uncertainties, is successfully applied to SVM. In robust-optimization-based SVM model, the bounded perturbation is added to the value of the input and the constraints keep feasible for all perturbation values in the uncertainty set, see \cite{svm:ro:Trafalis,svm:ro:Bertsimas,svm:ro:Singla,svm:ro:Pant} for details. However, the characteristics of the robust optimization may lead to over-conservatism.

Chance-constrained SVM handles uncertainties via controlling the probability of misclassification under a given probability distribution for each sample point. The chance-constrainted model is often faced with the challenge of computational tractability, so that the approximation is usually adopted for computing. For example, Peng et al.\cite{svm:cc:Peng} proposed a chance-constrained conic-segmentation SVM model and derived a mixed integer programming formulation via sample average approximation. 

Futhermore, if the true probability distribution of input data is unknown, distributionally robust chance-constrained (DRC) SVM ensures small probability of misclassification for all probability distribution in an ambiguity set. The moment-based ambiguity set is the most commonly used in DRC SVM models due to computational tractability. 
Shivaswamy et al. \cite{svm:drcc:Shivaswamy}, Huang et al. \cite{svm:drcc:Huang} and Khanjani-Shiraz et al.\cite{svm:drcc:KS} derived second order cone program (SOCP) reformulations using Chebyshev inequality. Bhadra et al. \cite{svm:drcc:Bhadra} and Ben-Tal et al. \cite{svm:drcc:Ben-Tal} utilized Bernstein bounds to derive less conservative SOCP reformulations. In addition to using probability inequalities, other reformulations from the optimization perspective are also derived in \cite{svm:drcc:Wang}. 

The metric-based ambiguity set, especially Wasserstein-distance-based ambiguity set, is widely used in machine learning\cite{ml:wdro}. For DRC SVM, Ma and Wang\cite{svm:wdrcc} studied $l_2$ Wasserstein ambiguity and transformed the DRC SVM model into a 0-1 mixed-integer SOCP model. Other Wasserstein-distance-based distributionally robust approach for SVM, that is expectation constraint instead of chance constraint, is investigated in \cite{svm:dro:Lee,svm:dro:Li}.

The classical SVM, as well as DRC SVM, are usually faced with the challenge that the computational complexity depends on the sample size. They often consider the small perturbation near every sample data in existing research of DRC SVM. Actually, it is natural to assume that the sample points with the same label are independent and identically distributed and estimate moments of the sample population for each class. The two classes of sample data may have different probability distributions with different supports and moments so that they are separable. 
Thus, we can formulate a new model for SVM, in which only one distributionally robust chance constraint is constructed for each class of data. The ambiguity set of each class is constructed based on all sample data belonging to that class. The dependence of the number of model constraints on the sample size is naturally eliminated and it efficiently handles the computational difficulties brought by big data. 

This paper studies a new moment-based DRC SVM model, called globalized distributionally robust chance-constrained (GDRC) SVM. The globalization means that we consider the distributional uncertainties of the sample population for each class, instead of the small perturbation around each sample point. 
To guarantee the robustness of the optimal solution, the ambiguity set is constructed using the confidence regions of sample moments. Core-set constraints are introduced to overcome the potential of over-conservatism in moment-based DRC SVM model. Additionally, the key to the classification hyperplane is always the set of support vectors, that is a set of data points near the classification hyperplane. Therefore, we can pay more attention to the set of data points near the potential classification hyperplane when constructing the core sets, to improve the classification quality. In the ambiguity set, an expected distance constraint from the random vector to core sets is constructed to control the degree of attention paid to the core sets. 

The main contributions of this paper are summarized as follows.
\begin{itemize}
\item A new globalized moment-based DRC SVM model is proposed, which focuses on the statistical properties of the sample population for each class, rather than the small perturbations around each sample point. It significantly reduces the number of constraints and eliminates the dependence of computational complexity on the sample size, thereby effectively addressing the challenges of big data.
\item In the GDRC SVM model, the core sets are introduced to capture some small regions near the potential classification hyperplane, in which the support vector may be contained. The core-set constraint in the ambiguity set not only helps to reduce the conservatism, but also helps to improve the classification accuracy. 
\item The GDRC SVM model can be equivalently reformulated as a semi-definite programming (SDP) model. In order to overcome the disadvantage of too long calculation time required for large-scale SDP, an approximation for GDRC SVM model is introduced to deal with classification of high dimensional datasets. The approximation is based on the principal component analysis (PCA), so choosing appropriate principal components may also help to improve the generalization ability.
\end{itemize}

The remainder of this paper is organized as follows. 
In Section \ref{sec:svm}, we introduce the distributionally robust chance-constrained SVM model and our globalized formulation. Section \ref{sec:ref-app} contains the equivalent reformulation of GDRC SVM and its approximation based on PCA. Section \ref{sec:num} presents the numerical experiments and results on both simulated data and real datasets. Section \ref{sec:cons} concludes this paper.

\emph{Notations.} $\Xi \subseteq \mathbb{R}^n$ is the sample space, $\mathcal{B}(\Xi)$ is the Borel $\sigma$-algebra on $\Xi$, and $\mathcal{P}(\Xi)$ is the set of all probability measures or distributions on measurable space $(\Xi,\mathcal{B}(\Xi))$. $E_F[\cdot]$ represents the expectation with respect to the probability distribution $F$. $\mathbb{R}^n$ is the set of real $n$-dimensional vectors; and $\mathbb{R}^n_+$ is the set of non-negative vectors in $\mathbb{R}^n$. $\mathbf{1}_n \in \mathbb{R}^n$ is a vector of 1, $\bm{I}_n \in \mathbb{R}^{n \times n}$ is the identity matrix, and $\bm{0}_{m,n} \in \mathbb{R}^{m \times n}$ is the zero matrix. For two real matrices $\bm{A},\bm{B} \in \mathbb{R}^{m\times n}$, the Frobenius inner product is defined as $\bm{A} \cdot \bm{B} = tr(\bm{A}^{\top}\bm{B})$. 
For proper convex function $g:\mathbb{R}^n \to \mathbb{R}$, we let 
$\mathrm{dom}(g)=\{\bm{x} \in \mathbb{R}^n \vert g(\bm{x}) < \infty \}$. 
The convex conjugate function of $g$ is defined as
$g^*(\bm{y})=\sup \limits_{\bm{x}\in \mathrm{dom}(g)} \{\bm{y}^{\top} \bm{x} -g(\bm{x})\}.$
For a function $f(\cdot,\cdot)$ with two variables, let $f^{**}(\cdot,\cdot)$ denote the convex conjugate function with respect to both variables. The indicator function for the set $S$ is defined as follows:
\begin{equation*}
   \delta(\bm{x}\vert S)=
   \left\{ \begin{aligned}  0 \qquad &\mathrm{if} \quad \bm{x} \in S, \\  
\infty \qquad & \mathrm{otherwise}. \end{aligned}\right.
\end{equation*}
The 0-1 indicator function for the set $S$ is defined as follows:
\begin{equation*}
   \mathds{1}_S(\bm{x})=
   \left\{ \begin{aligned}  1 \qquad &\mathrm{if} \quad \bm{x} \in S, \\  0 \qquad & \mathrm{otherwise}. \end{aligned}\right.
\end{equation*}

\section{Distributionally robust chance-constrained SVM}
\label{sec:svm}
Given a dataset of $N = N_1 + N_2$ sample points, it is composed of two classes of data $\{(\bm{x}_1^i,y_1)\}_{i=1}^{N_1} \cup \{(\bm{x}_2^i,y_2)\}_{i=1}^{N_2}$, where $y_1 =1$ and $y_2 =-1$ are the labels representing the classes, $k=1$ and $k=2$ represent the class $+1$ and $-1$ respectively, and $\bm{x}_k^i \in \mathbb{R}^n$ is the $i$th feature vector belonging to the $k$th class. It is linearly separable if there exists a hyperplane $\bm{w}^\top \bm{x} + b =0$ separating the sample points into two classes. The classification is determined by the sign of $\bm{w}^\top\bm{x} + b$, that is, $\bm{x}$ belongs to class $+1$ if $\bm{w}^\top\bm{x} + b > 0$, otherwise it belongs to class $-1$.
Linear SVM aims to find the maximum-margin hyperplane and it is usually formulated as 
\begin{equation}\label{eq:svm}
\begin{aligned}
\min \quad & \frac{1}{2} ||\bm{w}||_2^2 + C \left(\sum \limits_{i=1}^{N_1} \xi^i_1 + \sum \limits_{i=1}^{N_2} \xi^i_2 \right) \\
s.t. \quad & y_k(\bm{w}^\top \bm{x}^i_k + b) \geqslant 1-\xi^i_k,i=1,\cdots,N_k,k=1,2 \\
& \bm{w} \in \mathbb{R}^n,  b\in \mathbb{R}, \bm{\xi}_1 \in \mathbb{R}^{N_1}_+, \bm{\xi}_2 \in \mathbb{R}^{N_2}_+.
\end{aligned}
\end{equation}
The model (\ref{eq:svm}) is also called soft-margin SVM, because the data point is allowed to be  within-margin ($0 < \xi^i_k \leqslant 1$) or misclassified ($\xi^i_k>1$). 

When considering the uncertainties in the input data, $\tilde{\bm{x}}^i_k \in \mathbb{R}^n$ denotes the random vector of the data point that follows the distribution $F^i_k$. For each data point, the constraint is allowed to be violated with probability no greater than $\varepsilon$. Then the chance-constrained SVM is developed and studied. However, it is often difficult to obtain the true probability distribution of uncertain data and to solve the chance-constrained model. 
The idea of distributionally robust optimization is further introduced, and DRC SVM model is usually formulated as
\begin{equation}\label{eq:drcc_svm}
\begin{aligned}
\min \quad & \frac{1}{2} ||\bm{w}||_2^2 + C \left(\sum \limits_{i=1}^{N_1} \xi^i_1 + \sum \limits_{i=1}^{N_2} \xi^i_2 \right)  \\
s.t. \quad & \sup \limits_{F^i_k \in \mathcal{D}^i_k} \mathbb{P}_{F^i_k} \left\{ y_k(\bm{w}^\top \tilde{\bm{x}}^i_k + b) \leqslant 1-\xi^i_k \right\} \leqslant \varepsilon,i=1,\cdots,N_k,k=1,2 \\
& \bm{w} \in \mathbb{R}^n, b\in \mathbb{R}, \bm{\xi}_1 \in \mathbb{R}^{N_1}_+, \bm{\xi}_2 \in \mathbb{R}^{N_2}_+.
\end{aligned}
\end{equation}
When the ambiguity set $\mathcal{D}^i_k$ contains only one distribution, it reduces to the chance-constrained SVM. 
In DRC SVM model, the ambiguity set with known mean and covariance matrix is widely used, such as \cite{svm:drcc:Shivaswamy,svm:drcc:Wang,svm:drcc:Huang}.
\begin{equation}\label{eq:ambiguity_known}
   \mathcal{D}^i_k =
   \left\{
   F^i_k \in \mathcal{P}(\mathbb{R}^n)
   \middle\vert
   \begin{array}{lcl}
    & E_{F^i_k}[\tilde{\bm{x}}^i_k] = \bm{\mu}^i_k  \\
    & E_{F^i_k}[(\tilde{\bm{x}}^i_k-\bm{\mu}^i_k)(\tilde{\bm{x}}^i_k-\bm{\mu}^i_k)^{\top}] = \bm{\Sigma}^i_k \\
    & \mathbb{P}_{F^i_k}(\tilde{\bm{x}}^i_k \in \mathbb{R}^n) = 1
   \end{array}
   \right\}, \begin{array}{lcl} &i=1,\cdots,N_k, \\ &k=1,2.\end{array}
\end{equation}

In the models (\ref{eq:svm}) and (\ref{eq:drcc_svm}), the number of constraints is equal to the sample size $N$, so that the computational cost is high when the sample size is large. Hence, there are some works dedicated to dealing with this problem. To handling big datasets with unlabeled data, Huang et al. \cite{svm:drcc:Huang} partitioned the sample data using the kNN algorithm, then designed a merging algorithm to merge some subsets, and finally obtained $p$ subsets. Thus, they established a DRC model as (\ref{eq:drcc_svm}), replacing the number of constraints $N$ with the number of subsets $p$. It is an efficient method to reduce the computational complexity of DRC SVM. Therefore, we wonder how it would perform if there were only two subsets, i.e., constructing only one distributionally robust chance constraint for each class of data points.

In SVM model, if two classes of data points are separable, they should come from different distributions, in particular, they may have different support sets, means and variances. 
In this study, we assume that the data points with the same label are independently and identically distributed. It is a reasonable assumption in SVM. If the data with the same label is multimodal, the probability distribution for this class can be regarded as a mixed distribution and the multimodality can be characterized in the ambiguity set. In the context of big data, we can obtain good estimates of the sample moments, and it is natural to formulate the following globalized distributionally robust chance-constrained (GDRC) model.
\begin{equation}\label{eq:gdrcc_svm}
\begin{aligned}
\min \quad & \frac{1}{2} ||\bm{w}||_2^2 + C ( \xi_1 + \xi_2) \\
s.t. \quad  & \sup \limits_{F_k \in \mathcal{D}_k}  \mathbb{P}_{F_k}\{y_k(\bm{w}^\top \tilde{\bm{x}}_k + b) \leqslant 1-\xi_k\} \leqslant \varepsilon , k=1,2\\
& \bm{w} \in \mathbb{R}^n,  b\in \mathbb{R}, \xi_1,\xi_2 \in \mathbb{R}_+,
\end{aligned}
\end{equation}

We refer to the model (\ref{eq:gdrcc_svm}) as GDRC SVM. The globalization means that we care about the uncertainty of the probability distribution of each class. The ambiguity set is defined as
\begin{equation}\label{eq:ambiguity}
   \mathcal{D}_k =
   \left\{
   F_k \in \mathcal{P}(\Xi_k)
   \middle\vert
   \begin{array}{lcl}
    & (E_{F_k}[\tilde{\bm{x}}_k]-\bm{\mu}_k)^{\top}\bm{\Sigma}_k^{-1}(E_{F_k}[\tilde{\bm{x}}_k]-\bm{\mu}_k) \leqslant \gamma_{k1}  \\
    & E_{F_k}[(\tilde{\bm{x}}_k-\bm{\mu}_k)(\tilde{\bm{x}}_k-\bm{\mu}_k)^{\top}] \preceq \gamma_{k2} \bm{\Sigma}_k\\
    & E_{F_k}[\min \limits_{j=1,\cdots,m_k} \theta_{kj} \min \limits_{\tilde{\bm{x}}_k' \in Y_{kj}} \phi_{kj}(\tilde{\bm{x}}_k,\tilde{\bm{x}}_k')] \leqslant 1 \\
    & \mathbb{P}_{F_k}(\tilde{\bm{x}}_k \in \Xi_k) = 1
   \end{array}
   \right\}, \quad k=1,2. 
\end{equation}
For the $k$th class, $\Xi_k \subset \mathbb{R}^n$ is the sample space; $\bm{\mu}_{k}$ and $\bm{\Sigma}_k$ represent the sample mean and sample covariance matrix, respectively; $\gamma_{k1}$ and $\gamma_{k2}$ are parameters determined by the confidence level; $Y_{kj} \subset \Xi_k,j=1,\cdots,m_k$ are core sets; $\phi_{kj},j=1,\cdots,m_k$ are closed, jointly convex and non-negative functions, satisfying $\phi_{kj}(\tilde{\bm{x}},\tilde{\bm{x}})=0$ for all $\tilde{\bm{x}} \in \mathbb{R}^n$, which are used to measure the distance between $\tilde{\bm{x}}$ and $\tilde{\bm{x}}'$; $\theta_{kj},j=1,\cdots,m_k$ are some given constants to characterize the degree of attention paid to the core sets. The core set $Y_{kj}$ is defined as 
\begin{equation*}
Y_{kj} = \left\{\hat{\bm{x}}_{kj} + \bm{A}_{kj} \bm{\zeta} \middle \vert ||\bm{\zeta}||_{p_{kj}} \leqslant \sqrt{\bar{\gamma}_{kj}}\right\},
\end{equation*}
where $\hat{\bm{x}}_{kj}$ denotes the center of the core set, $\bm{A}_{kj}$ is the perturbation matrix, $\bm{\zeta}$ is the random vector of ``primitive uncertainties", and $\bar{\gamma}_{kj}$ describes the size of the core set. 

In the ambiguity set (\ref{eq:ambiguity}), the first two constraints capture the confidence region of the mean and covariance matrix. The third constraint, that is the core-set constraint, limits the minimum expected distance from the random vector to all core sets. Thus, it forces that the probability of that $\tilde{\bm{x}}_k$ is farther from $\cup_{j=1}^{m_k} Y_{kj}$ is smaller, thereby, it strengthens the influence of core sets. Especially, the parameter $\theta_{kj}$ can control the degree of importance of the core set $Y_{kj}$. 

The core-set constraint plays two important roles in the GDRC SVM model. Firstly, the GDRC model maybe over-conservative if the ambiguity set determined by only mean and covariance matrix constraints contains too many unrealistic distributions. The core-set constraint helps reduce the conservatism by restricting high-density distribution regions. Secondly, the core sets $Y_{kj} \subset \Xi_k,j=1,\cdots,m_k,k=1,2$ are constructed to capture some important regions that may have a greater impact on the classification hyperplane to help improve the classification quality. As we all know, the optimal classification hyperplane provided by the SVM model is determined only by the support vectors. Consider the Lagrangian of (\ref{eq:svm}):
\begin{equation*}
\mathcal{L} = \frac{1}{2} ||\bm{w}||_2^2 + C \sum \limits_{k=1}^2 \sum \limits_{i=1}^{N_k} \left( \xi_k^i + \alpha_k^i(1-\xi_k^i-y_k(\bm{w}^\top \bm{x}_k^i + b))- \beta_k^i \xi_k^i \right)
\end{equation*}
The KKT conditions give that
\begin{equation*}
\bm{w} = \sum \limits_{k=1}^2 \sum \limits_{i=1}^{N_k} \alpha_k^i y_k \bm{x}_k^i, \quad
\sum \limits_{k=1}^2 \sum \limits_{i=1}^{N_k}\alpha_k^i y_k = 0, 
\end{equation*}
\begin{equation*}
\alpha_k^i(1-\xi_k^i -y_k(\bm{w}^T \bm{x}_k^i + b)) = 0, \quad
\beta_k^i \xi_k^i = 0.
\end{equation*}
Note that $\bm{w}$ is a linear combination of sample points. However, not all of sample points contribute to the classification hyperplane. Actually, we only need to pay attention to $\bm{x}_k^i$ with corresponding $\alpha_k^i \neq 0$. According to the complementary conditions, 
\begin{equation*}
\alpha_k^i \neq 0 \Rightarrow y_k(\bm{w}^T \bm{x}_k^i + b) = 1-\xi_k^i ,
\end{equation*}
which means that $\bm{x}_k^i$ is either on the margin or misclassified. It indicates that the sample points near the potential classification hyperplane play a more important role in the process of finding the hyperplane. Therefore, we use core sets to capture the regions near the potential classification hyperplane. 

It is hard to say which is more robust or more conservative between the DRC model (\ref{eq:drcc_svm}) with the ambiguity set (\ref{eq:ambiguity_known}) and the GDRC model (\ref{eq:gdrcc_svm}) with the ambiguity set (\ref{eq:ambiguity}). The DRC model (\ref{eq:drcc_svm}) processes the perturbation around each sample point and it ususally adopts small covariance matrix in the ambiguity set (\ref{eq:ambiguity_known}) to avoid too large robust region. The GDRC model (\ref{eq:gdrcc_svm}) processes the perturbation of the probability distribution of all sample points belonging to the same class, and the confidence regions of the mean and covariance matrix are adopted in (\ref{eq:ambiguity}). The determination of the parameters $\gamma_{k1},\gamma_{k2}$ can refer to \cite{dro:Delage}. Additionally, in the GDRC SVM model, the core-set constraint is constructed to avoid over-conservatism and improve the classification quality. 

\section{Reformulation and approximation of GDRC SVM}\label{sec:ref-app}
\subsection{SDP reformulation of GDRC SVM}
Assume that $\Xi_k = \mathbb{R}^n$. It follows from \cite{gdro} that $\mathbb{R}^n$ is a good choice if there is no prior information about the support of the probability distribution. The core-set constraint characterizes the data distribution area that decision makers are more concerned about, thereby reducing the degree of conservatism that maybe caused by the sample space being the entire space. We have the following theorem.

\begin{theorem}\label{thm:ref}
Assume that $\bm{\mu}_{k} \in Y_{kj}$, 
the distributionally robust chance constraint 
\begin{equation}\label{eq:gdrcc}
\sup \limits_{F_k \in \mathcal{D}_k}  \mathbb{P}_{F_k}\{y_k(\bm{w}^\top \tilde{\bm{x}}_k + b) \leqslant 1-\xi_k\} \leqslant \varepsilon ,
\end{equation}
with the ambiguity set (\ref{eq:ambiguity}), can be equivalently reformulated as:
\begin{equation} \label{eq:gdrcc_ref}
\left\{
\begin{aligned}
&  t_k \!+ \!\bm{\Lambda}_k \cdot (\gamma_{k2} \bm{\Sigma}_k  \!+ \!\bm{\mu}_k \bm{\mu}_k^{\top}) \!+ \! \sqrt{\gamma_{k1}}||\bm{\Sigma}_k^{1/2} (\bm{q}_k + 2\bm{\Lambda}_k \bm{\mu}_k)||_2 \! + \! \bm{q}_k^{\top} \bm{\mu}_k  \!+ \! r_k \leqslant \varepsilon \tau_k \\
& \begin{bmatrix} \bm{\Lambda}_k & \frac{1}{2}(\bm{v}_{kj}+\bm{q}_k+ y_k \bm{w})\\[1em]
\frac{1}{2}(\bm{v}_{kj}+\bm{q}_k+y_k\bm{w})^{\top} & \begin{aligned} & t_k - \tau_k + y_k b - 1+\xi_k - \delta^*(\bm{v}_{kj}\vert Y_{kj}) \\ &  -(r_k \theta_{kj} \phi_{kj})^{**}(\bm{v}_{kj},-\bm{v}_{kj}) \end{aligned} \end{bmatrix} \succeq 0 \\[1em]
& \begin{bmatrix} \bm{\Lambda}_k & \frac{1}{2}(\bm{u}_{kj}+\bm{q}_k)\\[1em]
\frac{1}{2}(\bm{u}_{kj}+\bm{q}_k)^{\top} & t_k - \delta^*(\bm{u}_{kj}\vert Y_{kj}) -(r_k \theta_{kj} \phi_{kj})^{**}(\bm{u}_{kj},-\bm{u}_{kj})  \end{bmatrix} \succeq 0 \\
& \bm{v}_{kj},\bm{u}_{kj} \in \mathbb{R}^n,\quad j =1,\cdots,m_k,\\
& \bm{\Lambda}_k \in \mathbb{S}_+^n, \bm{q}_k \in \mathbb{R}^n,  t_k \in\mathbb{R}, r_k,\tau_k \in\mathbb{R}_+.  \\
\end{aligned}
\right.
\end{equation}
\end{theorem}

\begin{proof}
The worst-case probability can be rewritten as
\begin{equation}\label{eq:ldrcc}
\begin{split}
\sup \limits_{F_k \in \mathcal{D}_k} & \mathbb{P}_{F_k}\{y_k (\bm{w}^\top \tilde{\bm{x}}_k + b) \leqslant 1-\xi_k\}  \\
= \sup \limits_{\mu_k \in \mathcal{M}_+(\mathbb{R}^n)} & \int_{\mathbb{R}^n} \mathds{1}_{
\{y_k(\bm{w}^T \tilde{\bm{x}}_k + b) \leqslant 1-\xi_k\}} \text{d}\mu_k(\tilde{\bm{x}}_k) \\
s.t. \quad & \int_{\mathbb{R}^n} \text{d}\mu_k(\tilde{\bm{x}}_k)  = 1\\
& \int_{\mathbb{R}^n} \begin{bmatrix} \bm{\Sigma}_k & \tilde{\bm{x}}_k -\bm{\mu}_k \\(\tilde{\bm{x}}_k -\bm{\mu}_k)^\top & \gamma_{k1} \end{bmatrix} \text{d}\mu_k(\tilde{\bm{x}}_k) \succeq 0\\
& \int_{\mathbb{R}^n} (\tilde{\bm{x}}_k-\bm{\mu}_k)(\tilde{\bm{x}}_k-\bm{\mu}_k)^\top \text{d}\mu_k(\tilde{\bm{x}}_k)  \preceq \gamma_{k2} \bm{\Sigma}_k \\
& \int_{\mathbb{R}^n} \min \limits_{j=1,\cdots,m_k} \theta_{kj} \min \limits_{\tilde{\bm{x}}_k' \in Y_{kj}} \phi_{kj}(\tilde{\bm{x}}_k,\tilde{\bm{x}}_k') \text{d}\mu_k(\tilde{\bm{x}}_k)  \leqslant 1,\\
\end{split}
\end{equation}
where $\mathcal{M}_+(\mathbb{R}^n)$ denotes the set of all finite positive measures on $\mathbb{R}^n$. Similarly to the proof of \cite[Lemma 1]{dro:Delage}, the dual of (\ref{eq:ldrcc}) is written as

\begin{small}
\begin{subequations}
\begin{align}
\min \limits_{\bm{\Lambda}_k,\bm{q}_k,t_k,r_k}  & t_k \! + \! \bm{\Lambda}_k \! \cdot \! (\gamma_{k2} \bm{\Sigma}_k \! +\! \bm{\mu}_k \bm{\mu}_k^{\top})\!+ \!\sqrt{\gamma_{k1}}||\bm{\Sigma}_k^{1/2} (\bm{q}_k\! +\! 2\bm{\Lambda}_k \bm{\mu}_k)||_2 \!+\! \bm{q}_k^{\top} \bm{\mu}_k \!+\! r_k \\
s.t. \quad & \mathds{1}_{\{y_k(\bm{w}^\top\tilde{\bm{x}}_k  +  b) \leqslant 1-\xi_k\}}-  r_k  \min \limits_{j=1,\cdots,m_k} \theta_{kj} \min \limits_{\tilde{\bm{x}}_k' \in Y_{kj}} \phi_{kj}(\tilde{\bm{x}}_k,\tilde{\bm{x}}_k') \notag \\ 
& \qquad \qquad \qquad \qquad \qquad \qquad - \tilde{\bm{x}}_k^{\top}\bm{\Lambda}_k \tilde{\bm{x}}_k  - \bm{q}_k^{\top} \tilde{\bm{x}}_k \leqslant t_k \quad \forall \tilde{\bm{x}}_k \in \mathbb{R}^n, \label{eq:cons}\\
& \bm{\Lambda}_k \in \mathbb{S}_+^n,  \bm{q}_k \in \mathbb{R}^n,  t_k\in\mathbb{R},  r_k \in\mathbb{R}_+ .
\end{align}
\end{subequations}
\end{small}

It follows from \cite[Proposition 3.4]{Shapiro} that the strong duality holds and the optimal values are equal. Due to the non-negativity of $\phi_{kj}$ and semi-positive definiteness of $\bm{\Lambda}_k$, we know that (\ref{eq:cons}) is equivalent to 
{\small
\begin{subequations}
\begin{numcases}{} 
 \tilde{\bm{x}}_k^{\top}\bm{\Lambda}_k \tilde{\bm{x}}_k + \bm{q}_k^{\top} \tilde{\bm{x}}_k+ t_k + r_k \min \limits_{j=1,\cdots,m_k} \theta_{kj} \min \limits_{\tilde{\bm{x}}_k' \in Y_{kj}} \phi_{kj}(\tilde{\bm{x}}_k,\tilde{\bm{x}}_k') - 1 \geqslant 0 \notag \\
\qquad \qquad \qquad \qquad \qquad \qquad \qquad \qquad \qquad \forall \tilde{\bm{x}}_k \in \{\tilde{\bm{x}}_k \vert y_k(\bm{w}^\top \tilde{\bm{x}}_k + b) \leqslant 1-\xi_k\} \label{eq:cons1}\\
 \tilde{\bm{x}_k}^{\top}\bm{\Lambda}_k \tilde{\bm{x}}_k + \bm{q}_k^{\top} \tilde{\bm{x}}_k+ t_k + r_k \min \limits_{j=1,\cdots,m_k} \theta_{kj} \min \limits_{\tilde{\bm{x}}_k' \in Y_{kj}} \phi_{kj}(\tilde{\bm{x}}_k,\tilde{\bm{x}}_k') \geqslant 0 \quad \forall \tilde{\bm{x}}_k \in \mathbb{R}^n .
\end{numcases}
\end{subequations}
}
Notice that $f_0(\tilde{\bm{x}}_k) = \tilde{\bm{x}}_k^{\top}\bm{\Lambda} \tilde{\bm{x}}_k + \bm{q}^{\top} \tilde{\bm{x}}_k+ t + r \min \limits_{j=1,\cdots,m_k} \theta_{kj} \min \limits_{\tilde{\bm{x}}_k' \in Y_{kj}} \phi_{kj}(\tilde{\bm{x}}_k,\tilde{\bm{x}}_k') - 1$ is convex in $\tilde{\bm{x}}_k$, and $f_1(\tilde{\bm{x}}_k) = \bm{w}^\top \tilde{\bm{x}}_k + b - 1+\xi_k$ is linear in $\tilde{\bm{x}}_k$. According to Farkas theorem\cite[Theorem 21.1]{cov_ana}, the inequality (\ref{eq:cons1}) holds if and only if there exists $ \tau_k \geqslant 0$ such that
\begin{equation*}
\begin{aligned}
& \tilde{\bm{x}}_k^{\top}\bm{\Lambda}_k \tilde{\bm{x}}_k + \bm{q}_k^{\top} \tilde{\bm{x}}_k+ t_k + r_k \min \limits_{j=1,\cdots,m_k} \theta_{kj} \min \limits_{\tilde{\bm{x}}_k' \in Y_{kj}} \phi_{kj}(\tilde{\bm{x}}_k,\tilde{\bm{x}}_k') - 1 \\
& \qquad \qquad \qquad \qquad \qquad \qquad + \tau_k(y_k(\bm{w}^\top \tilde{\bm{x}}_k + b) - 1+\xi_k) \geqslant 0 \quad \forall \tilde{\bm{x}}_k \in \mathbb{R}^n.
\end{aligned}
\end{equation*}
Therefore, the distributionally robust chance constraint (\ref{eq:gdrcc}) is equivalent to the following series of inequalities.
{\small
\begin{subequations}\label{eq:redrcc}
\begin{numcases}{} 
 t_k \! +\bm{\Lambda}_k \! \cdot \! (\gamma_{k2} \bm{\Sigma}_k \! +\! \bm{\mu}_k \bm{\mu}_k^{\top})+ \sqrt{\gamma_{k1}}||\bm{\Sigma}_k^{1/2} (\bm{q}_k\! +\! 2\bm{\Lambda}_k \bm{\mu}_k)||_2 + \bm{q}_k^{\top} \bm{\mu}_k + r_k \leqslant \varepsilon \label{eq:redrcc_a} \\
 \tilde{\bm{x}}_k^{\top}\bm{\Lambda}_k \tilde{\bm{x}}_k + \bm{q}_k^{\top} \tilde{\bm{x}}_k+ t_k + r_k \min \limits_{j=1,\cdots,m_k} \theta_{kj} \min \limits_{\tilde{\bm{x}}_k' \in Y_{kj}} \phi_{kj}(\tilde{\bm{x}}_k,\tilde{\bm{x}}_k') - 1 \notag \\
 \qquad \qquad \qquad \qquad \qquad \qquad \quad + \tau_k(y_k(\bm{w}^\top \tilde{\bm{x}}_k + b) - 1+\xi_k) \geqslant 0 \quad \forall \tilde{\bm{x}}_k \in \mathbb{R}^n \label{eq:redrcc_b}\\
 \tilde{\bm{x}}_k^{\top}\bm{\Lambda}_k \tilde{\bm{x}}_k + \bm{q}_k^{\top} \tilde{\bm{x}}_k+ t_k + r_k \min \limits_{j=1,\cdots,m_k} \theta_{kj} \min \limits_{\tilde{\bm{x}}_k' \in Y_{kj}} \phi_{kj}(\tilde{\bm{x}}_k,\tilde{\bm{x}}_k') \geqslant 0 \quad \forall \tilde{\bm{x}}_k \in \mathbb{R}^n \\
 \bm{\Lambda}_k \in \mathbb{S}_+^n, \bm{q}_k \in \mathbb{R}^n,  t_k\in\mathbb{R},  r_k,\tau_k \in\mathbb{R}_+ .
\end{numcases}
\end{subequations}
}
To decouple the bilinear terms $\tau_k(y_k(\bm{w}^\top \tilde{\bm{x}}_k + b) - 1+\xi_k)$, we can perform variable substitutions in (\ref{eq:redrcc}), that is, $\bm{\Lambda}_k = \frac{1}{\tau_k} \bm{\Lambda}_k$, $\bm{q}_k= \frac{1}{\tau_k} \bm{q}_k$, $t_k = \frac{1}{\tau_k} t_k$, $r_k = \frac{1}{\tau_k} r_k$, $\tau_k = \frac{1}{\tau_k}$. Before that, it needs to verify that $\tau_k$ cannot be zero. If $\tau_k=0$, it follows from (\ref{eq:redrcc_b}) and $\bm{\mu}_k \in Y_{kj}$ that $ \bm{\mu}_k^{\top}\bm{\Lambda}_k \bm{\mu}_k + \bm{q}_k^{\top} \bm{\mu}_k+ t_k  - 1  \geqslant 0$. It follows from (\ref{eq:redrcc_a}) that $ \bm{\mu}_k^{\top}\bm{\Lambda}_k \bm{\mu}_k + \bm{q}_k^{\top} \bm{\mu}_k+ t_k  \leqslant \varepsilon $. Thus, it produces a contradiction with $0 < \varepsilon < 1$. Performing variable substitutions in (\ref{eq:redrcc}), the reformulation of (\ref{eq:gdrcc}) becomes
{\small
\begin{subequations}\label{eq:redrcc2}
\begin{numcases}{} 
 t_k \! +\bm{\Lambda}_k \! \cdot \! (\gamma_{k2} \bm{\Sigma}_k \! +\! \bm{\mu}_k \bm{\mu}_k^{\top})+ \sqrt{\gamma_{k1}}||\bm{\Sigma}_k^{1/2} (\bm{q}_k\! +\! 2\bm{\Lambda}_k \bm{\mu}_k)||_2 + \bm{q}_k^{\top} \bm{\mu}_k + r_k \leqslant \varepsilon \tau_k  \\
 \tilde{\bm{x}}_k^{\top}\bm{\Lambda}_k \tilde{\bm{x}}_k + \bm{q}_k^{\top} \tilde{\bm{x}}_k+ t_k + r_k \min \limits_{j=1,\cdots,m_k} \theta_{kj} \min \limits_{\tilde{\bm{x}}_k' \in Y_{kj}} \phi_{kj}(\tilde{\bm{x}}_k,\tilde{\bm{x}}_k') - \tau_k \notag \\
 \qquad \qquad \qquad \qquad \qquad \qquad \qquad \quad + y_k(\bm{w}^\top \tilde{\bm{x}}_k + b) - 1+\xi_k \geqslant 0 \quad \forall \tilde{\bm{x}}_k \in \mathbb{R}^n \label{eq:redrcc2_b}\\
 \tilde{\bm{x}}_k^{\top}\bm{\Lambda}_k \tilde{\bm{x}}_k + \bm{q}_k^{\top} \tilde{\bm{x}}_k+ t_k + r_k \min \limits_{j=1,\cdots,m_k} \theta_{kj} \min \limits_{\tilde{\bm{x}}' \in Y_{kj}} \phi_{kj}(\tilde{\bm{x}}_k,\tilde{\bm{x}}_k') \geqslant 0 \quad \forall \tilde{\bm{x}}_k \in \mathbb{R}^n \label{eq:redrcc2_c}\\
 \bm{\Lambda}_k \in \mathbb{S}_+^n,  \bm{q}_k \in \mathbb{R}^n, t_k\in\mathbb{R}, r_k,\tau_k \in\mathbb{R}_+ .
\end{numcases}
\end{subequations}
}
Notice that (\ref{eq:redrcc2_b}) is equivalent to 
\begin{equation*}
\begin{split}
& -\tilde{\bm{x}}_k^{\top}\bm{\Lambda}_k \tilde{\bm{x}}_k - (\bm{q}_k+ y_k \bm{w})^{\top} \tilde{\bm{x}}_k- t_k + \tau_k  - y_k b + 1-\xi_k \\
& \qquad \qquad  \qquad \qquad \qquad \quad \leqslant r_k \theta_{kj} \min \limits_{\tilde{\bm{x}}_k' \in Y_{kj}} \phi_{kj}(\tilde{\bm{x}}_k,\tilde{\bm{x}}_k') \quad \forall \tilde{\bm{x}}_k \in \mathbb{R}^n, j=1,\cdots,m_k.
\end{split}
\end{equation*}
According to \cite[Theorem 1]{gro}, it holds if and only if there exists $\bm{v}_{kj} \in \mathbb{R}^n$ that satisfies
\begin{equation*}
\begin{aligned}
&- \inf \limits_{\tilde{\bm{x}}_k} \{\tilde{\bm{x}}_k^{\top}\bm{\Lambda}_k \tilde{\bm{x}}_k + (\bm{q}_k+ y_k\bm{w} +\bm{v}_{kj})^{\top} \tilde{\bm{x}}_k+ t_k - \tau_k +y_k b - 1+\xi_k \}\\
& \qquad \qquad + \delta^*(\bm{v}_{kj}\vert Y_{kj}) +(r_k \theta_{kj} \phi_{kj})^{**}(\bm{v}_{kj},-\bm{v}_{kj}) \leqslant 0 \quad j =1,\cdots,m_k,
\end{aligned}
\end{equation*}
which is furthermore equivalent to 
\begin{equation*}\label{eq:semi-inf-cons}
\begin{bmatrix} \bm{\Lambda}_k & \frac{1}{2}(\bm{v}_{kj}+\bm{q}_k+y_k \bm{w})\\[1em]
\frac{1}{2}(\bm{v}_{kj}+\bm{q}_k+y_k\bm{w})^{\top} & \begin{aligned} & t_k - \tau_k + y_k b - 1+\xi_k - \delta^*(\bm{v}_{kj}\vert Y_{kj}) \\ &  -(r_k \theta_{kj} \phi_{kj})^{**}(\bm{v}_{kj},-\bm{v}_{kj}) \end{aligned} \end{bmatrix} \succeq 0  \quad j =1,\cdots,m_k. \\
\end{equation*}
It is similar for (\ref{eq:redrcc2_c}). Therefore, the distributionally robust chance constraint (\ref{eq:gdrcc}) is equivalent to the following series of constraints.
\begin{equation*} 
\left\{
\begin{aligned}
&  t_k \!+ \!\bm{\Lambda}_k \cdot (\gamma_{k2} \bm{\Sigma}_k  \!+ \!\bm{\mu}_k \bm{\mu}_k^{\top}) \!+ \! \sqrt{\gamma_{k1}}||\bm{\Sigma}_k^{1/2} (\bm{q}_k + 2\bm{\Lambda}_k \bm{\mu}_k)||_2 \! + \! \bm{q}_k^{\top} \bm{\mu}_k  \!+ \! r_k \leqslant \varepsilon \tau_k \\
& \begin{bmatrix} \bm{\Lambda}_k & \frac{1}{2}(\bm{v}_{kj}+\bm{q}_k+ y_k \bm{w})\\[1em]
\frac{1}{2}(\bm{v}_{kj}+\bm{q}_k+y_k\bm{w})^{\top} & \begin{aligned} & t_k - \tau_k + y_k b - 1+\xi_k - \delta^*(\bm{v}_{kj}\vert Y_{kj}) \\ &  -(r_k \theta_{kj} \phi_{kj})^{**}(\bm{v}_{kj},-\bm{v}_{kj}) \end{aligned} \end{bmatrix} \succeq 0 \\[1em]
& \begin{bmatrix} \bm{\Lambda}_k & \frac{1}{2}(\bm{u}_{kj}+\bm{q}_k)\\[1em]
\frac{1}{2}(\bm{u}_{kj}+\bm{q}_k)^{\top} & t_k - \delta^*(\bm{u}_{kj}\vert Y_{kj}) -(r_k \theta_{kj} \phi_{kj})^{**}(\bm{u}_{kj},-\bm{u}_{kj})  \end{bmatrix} \succeq 0 \\
& \bm{v}_{kj},\bm{u}_{kj} \in \mathbb{R}^n,\quad j =1,\cdots,m_k,\\
& \bm{\Lambda}_k \in \mathbb{S}_+^n, \bm{q}_k \in \mathbb{R}^n,  t_k \in\mathbb{R}, r_k,\tau_k \in\mathbb{R}_+.  \\
\end{aligned}
\right.
\end{equation*}
It completes the proof.
\end{proof}

The computations involving core sets $Y_{kj}$ and distance functions $\phi_{kj}$ are separated in (\ref{eq:gdrcc_ref}). For some choices of core sets and distance functions discussed in \cite{gdro}, the conjugate functions are easily calculated. In the following proposition, we show the tractable reformulation for a commonly used class of functions.

\begin{proposition}
Given 
\begin{equation}\label{eq:setting}
\begin{aligned}
& Y_{kj} = \left\{\hat{\bm{x}}_{kj} + \bm{A}_{kj} \bm{\zeta} \middle\vert ||\bm{\zeta}||_{p_{kj}} \leqslant \sqrt{\bar{\gamma}_{kj}}\right\},\\
& \phi_{kj}(\tilde{\bm{x}}_k,\tilde{\bm{x}}_k')=||\tilde{\bm{x}}_k-\tilde{\bm{x}}_k' ||_{p_{kj}}. 
\end{aligned}
\end{equation} 

Then the GDRC SVM model (\ref{eq:gdrcc_svm}) is equivalently reformulated as
\begin{equation}\label{eq:svm_gdrcc_norm}
\begin{aligned}
\min \quad & \frac{1}{2} ||\bm{w}||_2^2 + C (\xi_1 + \xi_2) \\
s.t. \quad  &  t_k \!+ \!\bm{\Lambda}_k \cdot (\gamma_{k2} \bm{\Sigma}_k  \!+ \!\bm{\mu}_k \bm{\mu}_k^{\top}) \!+ \! \sqrt{\gamma_{k1}}||\bm{\Sigma}_k^{1/2} (\bm{q}_k \! + \! 2\bm{\Lambda}_k \bm{\mu}_k)||_2 \! + \! \bm{q}_k^{T} \bm{\mu}_k  \!+ \! r_k \leqslant \varepsilon \tau_k \\
& \begin{bmatrix} \bm{\Lambda}_k & \frac{1}{2}(\bm{v}_{kj}+\bm{q}_k+ y_k \bm{w})\\[1em]
\frac{1}{2}(\bm{v}_{kj}+\bm{q}_k+y_k\bm{w})^{\top} & \begin{aligned} & t_k - \tau_k + y_k b - 1+\xi_k \\ &  -( \hat{\bm{x}}_{kj} \bm{v}_{kj}+\sqrt{\bar{\gamma}_{kj}}||\bm{A}_{kj}^T \bm{v}_{kj}||_{q_{kj}}) \end{aligned} \end{bmatrix} \succeq 0 \\[1em]
& \begin{bmatrix} \bm{\Lambda}_k & \frac{1}{2}(\bm{u}_{kj}+\bm{q}_k)\\[1em]
\frac{1}{2}(\bm{u}_{kj}+\bm{q}_k)^{\top} & t_k - ( \hat{\bm{x}}_{kj} \bm{u}_{kj}+\sqrt{\bar{\gamma}_{kj}}||\bm{A}_{kj}^T \bm{u}_{kj}||_{q_{kj}}) \end{bmatrix} \succeq 0 \\
& ||\bm{v}_{kj}||_{q_{kj}} \leqslant r_k \theta_{kj}, \quad ||\bm{u}_{kj}||_{q_{kj}} \leqslant r_k \theta_{kj}, \\
& \bm{v}_{kj},\bm{u}_{kj} \in \mathbb{R}^n,\quad j =1,\cdots,m_k,\\
& \bm{\Lambda}_k \in \mathbb{S}_+^n, \bm{q}_k \in \mathbb{R}^n,  t_k \in\mathbb{R}, r_k,\tau_k \in\mathbb{R}_+, \quad k=1,2 \\
& \bm{w} \in \mathbb{R}^n, b\in \mathbb{R},  \xi_1 ,\xi_2 \in \mathbb{R}_+.
\end{aligned}
\end{equation}
where $1/q_{kj} + 1/p_{kj} = 1$. If all norms are $l_1$, $l_2$ or $l_{\infty}$ norms, the reformulation (\ref{eq:svm_gdrcc_norm}) is SDP representable.
\end{proposition}
\begin{proof}
It is easy to verify that
\begin{equation*}
\begin{aligned}
& \delta^*(\bm{v}_{kj}\vert Y_{kj})  = \hat{\bm{x}}_{kj} \bm{v}_{kj}+\sqrt{\bar{\gamma}_{kj}}||\bm{A}_{kj}^\top\bm{v}_{kj}||_{q_{kj}}, \\
&(r_k\theta_{kj} \phi_{kj})^{**}(\bm{v}_{kj},-\bm{v}_{kj}) = \left\{\begin{aligned}
      & 0 \quad \text{if} \quad ||\bm{v}_{kj}||_{q_{kj}} \leqslant r_k\theta_{kj}, \\
      & \infty \quad \text{otherwise},
      \end{aligned}
      \right.
\end{aligned}
\end{equation*}
where $1/p_{kj}+1/q_{kj}=1$. Then the result follows from Theorem \ref{thm:ref} and substituting them in GDRC SVM model (\ref{eq:gdrcc_svm}).
\end{proof}

\subsection{Approximation based on principal component analysis}\label{sec:app}
In the previous subsection, we show that the GDRC SVM model can be equivalently reformulated as an SDP model with appropriate choices of core sets $Y_{kj}$ and distance functions $\phi_{kj}$. However, it is well known that solving the large-scale SDP model requires high computational cost when using interior point methods. Therefore, some algorithms are developed to deal with large-scale SDPs\cite{sdpnal}, and some approximations are proposed for DRO models\cite{dro:pca}. In order to cope with the large-scale cases, we applied the approximation method based on principal component analysis (PCA) proposed by Cheng et al.\cite{dro:pca} to GDRC SVM. 

Consider the eigenvalue decomposition on $\bm{\Sigma}_k$, that is $\bm{\Sigma}_k = \bm{U}_k\bm{D}_k\bm{U}_k^\top$ where $\bm{U}_k$ is an orthogonal matrix and $\bm{D}_k$ is a diagonal matrix with diagonal elements $\lambda_1 \geqslant \dots \geqslant \lambda_n >0 $. Let $\bm{D}_k^{1/2} = \text{diag} \{\sqrt{\lambda_1},\dots,\sqrt{\lambda_n} \} $ and denote $\bm{S}_k = \bm{U}_k\bm{D}_k^{1/2}$. By variable substitution $\tilde{\bm{z}}_k  = \bm{S}_k^{-1}(\tilde{\bm{x}}_k-\bm{\mu}_k)$, it is easy to verify that

\begin{equation}\label{eq:drcc_evd}
\begin{aligned}
& \sup \limits_{F_k \in \mathcal{D}_k}  \mathbb{P}_{F_k}\left\{y_k(\bm{w}^\top \tilde{\bm{x}}_k + b) \leqslant 1-\xi_k\right\}  \\
= & \sup \limits_{F^{(n)}_k \in \mathcal{D}^{(n)}_{k}}  \mathbb{P}_{F^{(n)}_k}\left\{y_k \left(\bm{w}^\top(\bm{\mu}_k + \bm{S}_k\tilde{\bm{z}}_k) + b \right) \leqslant 1-\xi_k \right\},
\end{aligned}
\end{equation}

where 
\begin{equation}\label{eq:ambiguity_evd}
   \mathcal{D}^{(n)}_k =
   \left\{
   F^{(n)}_k \in \mathcal{P}(\mathbb{R}^n)
   \middle\vert
   \begin{array}{lcl}
    & E_{F^{(n)}_k}[\tilde{\bm{z}}_k]^{\top} E_{F^{(n)}_k}[\tilde{\bm{z}}_k] \leqslant \gamma_{k1}  \\
    & E_{F^{(n)}_k}[\tilde{\bm{z}}_k \tilde{\bm{z}}_k^{\top}] \preceq \gamma_{k2} \bm{I}_n \\
    & E_{F^{(n)}_k}[\min \limits_{j=1,\cdots,m_k} \theta_{kj} \min \limits_{\tilde{\bm{z}}_k' \in Z_{kj}} \psi_{kj}(\tilde{\bm{z}}_k,\tilde{\bm{z}}_k')] \leqslant 1 \\
    & \mathbb{P}_{F^{(n)}_k}(\tilde{\bm{z}}_k \in \mathbb{R}^n) = 1
   \end{array}
   \right\},
\end{equation}
\begin{equation}\label{eq:core_evd}
\begin{aligned}
& Z_{kj} = \left\{\tilde{\bm{z}}_k \in \mathbb{R}^n \middle\vert \bm{\mu}_k + \bm{S}_k\tilde{\bm{z}}_k \in Y_{kj}\right\}, \\
\text{and } & \psi_{kj}(\tilde{\bm{z}}_k, \tilde{\bm{z}}_k')= \phi_{kj}(\bm{\mu}_k + \bm{S}_k \tilde{\bm{z}}_k,\bm{\mu}_k + \bm{S}_k \tilde{\bm{z}}_k').
\end{aligned}
\end{equation}

To derive a low-rank approximation, $\tilde{\bm{x}}_k^{(r)}  = \bm{\mu}_k + \bm{S}_k^{(r)}\tilde{\bm{z}}_k^{(r)}$ is adopted to approximate $\tilde{\bm{x}}_k$, where $\bm{S}_k^{(r)} = \bm{U}_{k(n\times r)}\bm{D}_{k(r)}^{1/2}$, $\bm{U}_{k(n\times r)}$ is the $n \times r$ upper left submatrix of $\bm{U}_k$ and $\bm{D}_{k(r)}$ is the $r \times r$ upper left submatrix of $\bm{D}_k$, and $\tilde{\bm{z}}_k^{(r)}$ is the vector of the first $r$ elements of $\tilde{\bm{z}}_k$. Hence, the distributionally robust chance constraint (\ref{eq:gdrcc}) is approximated by
\begin{equation}\label{eq:drcc_app}
\sup \limits_{F^{(r)}_k \in \mathcal{D}^{(r)}_{k}}  \mathbb{P}_{F^{(r)}_k} \left\{y_k \left(\bm{w}^\top (\bm{\mu}_k + \bm{S}_k^{(r)}\tilde{\bm{z}}_k^{(r)}) + b \right) \leqslant 1-\xi_k \right\} \leqslant \varepsilon,
\end{equation}
where 
\begin{equation}\label{eq:ambiguity_evd_app}
   \mathcal{D}^{(r)}_k =
   \left\{
   F^{(r)}_k \in \mathcal{P}(\mathbb{R}^r)
   \middle\vert
   \begin{array}{lcl}
    & E_{F^{(r)}_k}[\tilde{\bm{z}}_k^{(r)}]^{\top} E_{F^{(r)}_k}[\tilde{\bm{z}}_k^{(r)}] \leqslant \gamma_{k1}  \\
    & E_{F^{(r)}_k}[\tilde{\bm{z}}_k^{(r)} (\tilde{\bm{z}}_k^{(r)})^{\top}] \preceq \gamma_{k2} \bm{I}_r \\
    & E_{F^{(r)}_k}[\min \limits_{j=1,\cdots,m_k} \theta_{kj} \min \limits_{\tilde{\bm{z}}_k'^{(r)} \in Z_{kj}^{(r)}} \psi_{kj}(\tilde{\bm{z}}_k^{(r)},\tilde{\bm{z}}_k'^{(r)})] \leqslant 1 \\
    & \mathbb{P}_{F^{(r)}_k}(\tilde{\bm{z}}_k^{(r)} \in \mathbb{R}^r) = 1
   \end{array}
   \right\},
\end{equation}
\begin{equation}\label{eq:core_app}
\begin{aligned}
& Z_{kj}^{(r)} = \left\{\tilde{\bm{z}}_k^{(r)} \in \mathbb{R}^r \middle \vert \bm{\mu}_k + \bm{S}_k^{(r)}\tilde{\bm{z}}_k^{(r)} \in Y_{kj}\right\}, \\
\text{and } & \psi_{kj}(\tilde{\bm{z}}_k^{(r)}, \tilde{\bm{z}}_k'^{(r)})= \phi_{kj}(\bm{\mu}_k + \bm{S}_k^{(r)}\tilde{\bm{z}}_k^{(r)},\bm{\mu}_k + \bm{S}_k^{(r)}\tilde{\bm{z}}_k'^{(r)}).
\end{aligned}
\end{equation}

It is obvious that when $r=n$, (\ref{eq:drcc_app}), (\ref{eq:ambiguity_evd_app}) and (\ref{eq:core_app}) coincide with (\ref{eq:drcc_evd}), (\ref{eq:ambiguity_evd}) and (\ref{eq:core_evd}), respectively. The approximation maintains the same structure of the original model when reformulating it, as shown by the following two propositions.

\begin{proposition}
Assume that $\bm{\mu}_{k} \in Y_{kj}$, 
the the distributionally robust chance constraint (\ref{eq:drcc_app}) with the ambiguity set (\ref{eq:ambiguity_evd_app}) can be reformulated as:
\begin{equation} \label{eq:gdrcc_app_ref}
\left\{
\begin{aligned}
&  t_k + \gamma_{k2}tr(\bm{\Lambda}_k) + \sqrt{\gamma_{k1}}||\bm{q}_k ||_2 +  r_k \leqslant \varepsilon \tau_k \\
& \begin{bmatrix} \bm{\Lambda}_k & \frac{1}{2}\left(\bm{v}_{kj}+\bm{q}_k+ y_k  \bm{S}_k^{(r)\top} \bm{w}\right)\\[1em]
\frac{1}{2}\left(\bm{v}_{kj}+\bm{q}_k+y_k \bm{S}_k^{(r)\top} \bm{w}\right)^{\top} & \begin{aligned} & t_k - \tau_k + y_k(\bm{w}^{\top} \bm{\mu}_k + b) - 1+\xi_k  \\ & - \delta^*(\bm{v}_{kj}\vert Z^{(r)}_{kj}) -(r_k \theta_{kj} \psi_{kj})^{**}(\bm{v}_{kj},-\bm{v}_{kj}) \end{aligned} \end{bmatrix} \succeq 0 \\[1em]
& \begin{bmatrix} \bm{\Lambda}_k & \frac{1}{2}(\bm{u}_{kj}+\bm{q}_k)\\[1em]
\frac{1}{2}(\bm{u}_{kj}+\bm{q}_k)^{\top} & t_k - \delta^*(\bm{u}_{kj}\vert Z^{(r)}_{kj}) -(r_k \theta_{kj} \psi_{kj})^{**}(\bm{u}_{kj},-\bm{u}_{kj})  \end{bmatrix} \succeq 0 \\
& \bm{v}_{kj},\bm{u}_{kj} \in \mathbb{R}^r, \quad j =1,\cdots,m_k ,\\
& \bm{\Lambda}_k \in \mathbb{S}_+^r,  \bm{q}_k \in \mathbb{R}^r,  t_k \in\mathbb{R},  r_k,\tau_k \in\mathbb{R}_+.
\end{aligned}
\right.
\end{equation}
\end{proposition}
\begin{proof}
It directly follows from Theorem \ref{thm:ref}.
\end{proof}

\begin{proposition}
Given 
\begin{equation}
\begin{split}
& Y_{kj} = \left\{\hat{\bm{x}}_{kj} + \bm{A}_{kj} \bm{\zeta} \middle\vert ||\bm{\zeta}||_{p_{kj}} \leqslant \sqrt{\bar{\gamma}_{kj}}\right\},\\
& \phi_{kj}(\tilde{\bm{x}},\tilde{\bm{x}}')=||\tilde{\bm{x}}-\tilde{\bm{x}}' ||_{p_{kj}}. 
\end{split}
\end{equation} 
Then (\ref{eq:svm_gdrcc_norm}) is approximated by
\begin{equation}\label{eq:svm_gdrcc_norm_app}
\begin{aligned}
\min \quad & \frac{1}{2} ||\bm{w}||_2^2 + C (\xi_1 + \xi_2) \\
s.t. \quad  &  t_k + \gamma_{k2}tr(\bm{\Lambda}_k) + \sqrt{\gamma_{k1}}||\bm{q}_k ||_2 +  r_k \leqslant \varepsilon \tau_k \\
& \begin{bmatrix} \bm{\Lambda}_k & \frac{1}{2}\left(\bm{q}_k+ \bm{S}_k^{(r)\top} (y_k \bm{w}+\bm{v}_{kj})\right)\\[1em]
\frac{1}{2}\left(\bm{q}_k+ \bm{S}_k^{(r)\top} (y_k \bm{w}+\bm{v}_{kj})\right)^{\top} & \begin{aligned} & t_k - \tau_k + y_k(\bm{w}^{\top} \bm{\mu}_k + b) - 1+\xi_k  \\ & - (\hat{\bm{x}}_{kj} - \bm{\mu}_k)^\top\bm{v}_{kj} -\sqrt{\bar{\gamma}_{kj}}||\bm{A}_{kj}^\top\bm{v}_{kj}||_{q_{kj}} \end{aligned} \end{bmatrix} \succeq 0 \\[1em]
& \begin{bmatrix} \bm{\Lambda}_k & \frac{1}{2}\left(\bm{q}_k+\bm{S}_k^{(r)\top}\bm{u}_{kj} \right)\\[1em]
\frac{1}{2}\left(\bm{q}_k+\bm{S}_k^{(r)\top}\bm{u}_{kj}\right)^{\top} & t_k - (\hat{\bm{x}}_{kj} - \bm{\mu}_k)^\top \bm{u}_{kj} -\sqrt{\bar{\gamma}_{kj}}||\bm{A}_{kj}^\top\bm{u}_{kj}||_{q_{kj}}  \end{bmatrix} \succeq 0 \\
& ||\bm{v}_{kj}||_{q_{kj}} \leqslant r_k \theta_{kj}, \quad ||\bm{u}_{kj}||_{q_{kj}} \leqslant r_k \theta_{kj} \\
& \bm{v}_{kj},\bm{u}_{kj} \in \mathbb{R}^n ,\quad j =1,\cdots,m_k ,\\
& \bm{\Lambda}_k \in \mathbb{S}_+^r,  \bm{q}_k \in \mathbb{R}^r,  t_k \in\mathbb{R}, r_k,\tau_k \in\mathbb{R}_+, \quad k=1,2 \\
& \bm{w} \in \mathbb{R}^n,  b\in \mathbb{R},  \xi_1 ,\xi_2 \in \mathbb{R}_+.
\end{aligned}
\end{equation}
where $1/q_{kj} + 1/p_{kj} = 1$. 
Especially, the optimal values of (\ref{eq:svm_gdrcc_norm_app}) and (\ref{eq:svm_gdrcc_norm}) are equal when $r=n$.
\end{proposition}
\begin{proof}
For $\bm{\omega}_{kj}\in \mathbb{R}^r$,
\begin{equation*}\label{eq:delta1}
\begin{aligned}
\delta^*(\bm{\omega}_{kj}\vert Z_{kj}^{(r)}) 
& = \sup \left\{\bm{\omega}_{kj}^\top \tilde{\bm{z}}_k^{(r)} \middle\vert \tilde{\bm{z}}_k^{(r)} \in Z_{kj}^{(r)}\right\}\\
& = \sup \left\{\bm{\omega}_{kj}^\top \tilde{\bm{z}}_k^{(r)}\middle\vert \bm{\mu}_k + \bm{S}_k^{(r)}\tilde{\bm{z}}_k^{(r)} \in Y_{kj} \right\}\\
& = \sup \left\{\bm{\omega}_{kj}^\top \tilde{\bm{z}}_k^{(r)} \middle\vert \bm{\mu}_k + \bm{S}_k^{(r)}\tilde{\bm{z}}_k^{(r)} = \hat{\bm{x}}_{kj} + \bm{A}_{kj} \bm{\zeta}, ||\bm{\zeta}||_{p_{kj}} \leqslant \sqrt{\bar{\gamma}_{kj}} \right\} \\
& = \inf \left\{(\hat{\bm{x}}_{kj} - \bm{\mu}_k)^\top \bm{v}_{kj} +  \sqrt{\bar{\gamma}_{kj}}||\bm{A}_{kj}^T\bm{v}_{kj} ||_{q_{kj}} \middle\vert \bm{\omega}_{kj} = \bm{S}_k^{(r)\top} \bm{v}_{kj} \right\}, \\
\end{aligned}
\end{equation*}
where the last equality follows from Lagrangian duality.

\begin{equation*}
\begin{aligned}
&(r_k \theta_{kj} \psi_{kj})^{**}(\bm{\omega}_{kj},-\bm{\omega}_{kj})\\
 = &\sup \left\{ \bm{\omega}_{kj}^\top(\tilde{\bm{z}}_k^{(r)} - \tilde{\bm{z}}_k'^{(r)}) - r_k \theta_{kj}\psi_{kj}(\tilde{\bm{z}}_k^{(r)},\tilde{\bm{z}}_k'^{(r)}) \right\} \\
 = &\sup \left\{ \bm{\omega}_{kj}^\top(\tilde{\bm{z}}_k^{(r)} - \tilde{\bm{z}}_k'^{(r)}) - r_k \theta_{kj}\phi_{kj}(\tilde{\bm{x}}_k^{(r)},\tilde{\bm{x}}_k'^{(r)}) \middle \vert \begin{array}{c} \tilde{\bm{x}}_k^{(r)}= \bm{\mu}_k + \bm{S}_k^{(r)}\tilde{\bm{z}}_k^{(r)}\\ \tilde{\bm{x}}_k'^{(r)}= \bm{\mu}_k + \bm{S}_k^{(r)}\tilde{\bm{z}}_k'^{(r)} \end{array} \right\} \\
 = &\inf \left\{ (r_k \theta_{kj}\phi_{kj})^{**}(\bm{\lambda}_{kj,1},\bm{\lambda}_{kj,2}) - \bm{\mu}_k^\top (\bm{\lambda}_{kj,1}+\bm{\lambda}_{kj,2}) \middle \vert \begin{array}{c} \bm{\omega}_{kj} - \bm{S}_k^{(r)\top} \bm{\lambda}_{kj,1} = 0\\ \bm{\omega}_{kj} + \bm{S}_k^{(r)\top} \bm{\lambda}_{kj,2} = 0 \end{array} \right\} \\
= & \inf \left\{ (r_k \theta_{kj}\phi_{kj})^{**}(\bm{\lambda}_{kj},-\bm{\lambda}_{kj}) \middle \vert \bm{\omega}_{kj} = \bm{S}_k^{(r)\top} \bm{\lambda}_{kj} \right\} \\
= & \left\{\begin{aligned}
      & 0 \quad \text{if} \quad \bm{\omega}_{kj} = \bm{S}_k^{(r)\top} \bm{\lambda}_{kj},||\bm{\lambda}_{kj}||_{q_{kj}} \leqslant r_k\theta_{kj}, \\
      & \infty \quad \text{otherwise}.
      \end{aligned}
      \right.
\end{aligned}
\end{equation*}
where the third equality follows from Lagrangian duality; the forth equality follows from the fact that $(r_k \theta_{kj}\phi_{kj})^{**}(\bm{\lambda}_{kj,1},\bm{\lambda}_{kj,2}) = \infty$ if $\bm{\lambda}_{kj,1}\neq -\bm{\lambda}_{kj,2}$. 

Substitute them in (\ref{eq:gdrcc_app_ref}), it gives that
\begin{equation} \label{eq:gdrcc_app_ref_norm}
\left\{
\begin{aligned}
&  t_k + \gamma_{k2}tr(\bm{\Lambda}_k) + \sqrt{\gamma_{k1}}||\bm{q}_k ||_2 +  r_k \leqslant \varepsilon \tau_k \\
& \begin{bmatrix} \bm{\Lambda}_k & \frac{1}{2}\left(\bm{\omega}_{kj}+\bm{q}_k+ y_k  \bm{S}_k^{(r)\top} \bm{w}\right)\\[1em]
\frac{1}{2}\left(\bm{\omega}_{kj}+\bm{q}_k+y_k \bm{S}_k^{(r)\top} \bm{w} \right)^{\top} & \begin{aligned} & t_k - \tau_k + y_k(\bm{w}^{\top} \bm{\mu}_k + b) - 1+\xi_k  \\ & - (\hat{\bm{x}}_{kj} - \bm{\mu}_k)^\top \bm{v}_{kj} - \sqrt{\bar{\gamma}_{kj}}||\bm{A}_{kj}^\top\bm{v}_{kj} ||_{q_{kj}} \end{aligned} \end{bmatrix} \succeq 0 \\[1em]
& \begin{bmatrix} \bm{\Lambda}_k & \frac{1}{2}(\bm{\beta}_{kj}+\bm{q}_k)\\[1em]
\frac{1}{2}(\bm{\beta}_{kj}+\bm{q}_k)^{\top} & t_k - (\hat{\bm{x}}_{kj} - \bm{\mu}_k)^\top \bm{u}_{kj} -  \sqrt{\bar{\gamma}_{kj}}||\bm{A}_{kj}^\top \bm{u}_{kj} ||_{q_{kj}}  \end{bmatrix} \succeq 0 \\
& \bm{\omega}_{kj} = \bm{S}_k^{(r)\top} \bm{v}_{kj},\quad \bm{\beta}_{kj} = \bm{S}_k^{(r)\top} \bm{u}_{kj} \\
& \bm{\omega}_{kj} = \bm{S}_k^{(r)\top} \bm{\lambda}_{kj},\quad \bm{\beta}_{kj} = \bm{S}_k^{(r)\top} \bm{\eta}_{kj}\\
& ||\bm{\lambda}_{kj}||_{q_{kj}} \leqslant r_k\theta_{kj}, \quad ||\bm{\eta}_{kj}||_{q_{kj}} \leqslant r_k\theta_{kj} \\
& \bm{\omega}_{kj},\bm{\beta}_{kj} \in \mathbb{R}^r,\bm{v}_{kj},\bm{u}_{kj},\bm{\lambda}_{kj},\bm{\eta}_{kj} \in \mathbb{R}^n, \quad j =1,\cdots,m_k ,\\
& \bm{\Lambda}_k \in \mathbb{S}_+^r,  \bm{q}_k \in \mathbb{R}^r,  t_k \in\mathbb{R},  r_k,\tau_k \in\mathbb{R}_+.
\end{aligned}
\right.
\end{equation}

When $r=n$, $\bm{S}_k^{(r)}= \bm{S}_k$ is invertible, so that $\bm{\lambda}_{kj} = \bm{v}_{kj}$, $\bm{\eta}_{kj}=\bm{u}_{kj}$. It inspires a tighter approximation when $r<n$. Let $\bm{\lambda}_{kj} = \bm{v}_{kj}$, $\bm{\eta}_{kj}=\bm{u}_{kj}$ and remove redundant equality constraints in (\ref{eq:gdrcc_app_ref_norm}), then it gives that:
\begin{equation}\label{gdrcc_app_ref_norm_t}
\left\{
\begin{aligned}
&  t_k + \gamma_{k2}tr(\bm{\Lambda}_k) + \sqrt{\gamma_{k1}}||\bm{q}_k ||_2 +  r_k \leqslant \varepsilon \tau_k \\
& \begin{bmatrix} \bm{\Lambda}_k & \frac{1}{2}\left(\bm{q}_k+ \bm{S}_k^{(r)\top} (y_k \bm{w}+\bm{v}_{kj})\right)\\[1em]
\frac{1}{2}\left(\bm{q}_k+ \bm{S}_k^{(r)\top} (y_k \bm{w}+\bm{v}_{kj})\right)^{\top} & \begin{aligned} & t_k - \tau_k + y_k(\bm{w}^{\top} \bm{\mu}_k + b) - 1+\xi_k  \\ & - (\hat{\bm{x}}_{kj} - \bm{\mu}_k)^\top\bm{v}_{kj} -\sqrt{\bar{\gamma}_{kj}}||\bm{A}_{kj}^\top \bm{v}_{kj}||_{q_{kj}} \end{aligned} \end{bmatrix} \succeq 0 \\[1em]
& \begin{bmatrix} \bm{\Lambda}_k & \frac{1}{2}\left(\bm{q}_k+\bm{S}_k^{(r)\top}\bm{u}_{kj}\right)\\[1em]
\frac{1}{2}\left(\bm{q}_k+\bm{S}_k^{(r)\top}\bm{u}_{kj}\right)^{\top} & t_k - (\hat{\bm{x}}_{kj} - \bm{\mu}_k)^\top \bm{u}_{kj} -\sqrt{\bar{\gamma}_{kj}}||\bm{A}_{kj}^\top \bm{u}_{kj}||_{q_{kj}}  \end{bmatrix} \succeq 0 \\
& ||\bm{v}_{kj}||_{q_{kj}} \leqslant r_k \theta_{kj}, \quad ||\bm{u}_{kj}||_{q_{kj}} \leqslant r_k \theta_{kj} \\
& \bm{v}_{kj},\bm{u}_{kj} \in \mathbb{R}^n ,\quad j =1,\cdots,m_k ,\\
& \bm{\Lambda}_k \in \mathbb{S}_+^r, \bm{q}_k \in \mathbb{R}^r,  t_k \in\mathbb{R}, r_k,\tau_k \in\mathbb{R}_+, \quad k=1,2. \\
\end{aligned}
\right.
\end{equation}
Addtitionally, it is easy to verify that (\ref{gdrcc_app_ref_norm_t}) is equivalent to 
\begin{equation}
\sup \limits_{F \in \mathcal{D}^{(n)}_{k}}  \mathbb{P}_{F^{(n)}_{k}} \left\{y_k \left(\bm{w}^\top (\bm{\mu}_k + \bm{S}_k\tilde{\bm{z}}_k) + b \right) \leqslant 1-\xi_k \right\} \leqslant \varepsilon,
\end{equation}
when $r=n$, and it is a relaxed approximation when $r<n$. Hence, the result follows from substituting it in the GDRC SVM model (\ref{eq:gdrcc_svm}).
\end{proof}

\subsection{Aproximation quality}
For convenience of writing, denote the approximation model (\ref{eq:svm_gdrcc_norm_app}) as $M_r$ with the optimal value $v^*(r)$. Similarly to that in \cite{dro:pca}, we can derive an upper bound for the gap of optimal values $v^*(n)-v^*(r)$.

\begin{proposition}\label{prop:app_quality}
Suppose that $(\bm{w}^{*},b^*,\bm{\xi}^{*},\bm{\Lambda}_k^*,\bm{q}_k^*,\bm{v}_{kj}^*,\bm{u}_{kj}^*,t_k^*,r_k^*,\tau_k^*)$ is a set of optimal solutions of the model $M_r$, then 
\begin{equation}\label{eq:gap}
\begin{aligned}
v^*(n)-v^*(r) \leqslant  \frac{C}{2} \sum \limits_{k,j}  (||\bm{g}_{kj}||_2 + ||\bm{h}_{kj}||_2),
\end{aligned}
\end{equation}
where $\bm{g}_{kj} = \bm{S}_k^{(n-r)\top}(y_k \bm{w}^* + \bm{v}_{kj}^*)$, $\bm{h}_{kj} = \bm{S}_k^{(n-r)\top} \bm{u}_{kj}^*$, $\bm{S}_k^{(n-r)} = \bm{U}_{k(n\times (n-r))} \bm{D}_{k(n-r)}^{1/2}$, $\bm{U}_{k(n\times (n-r))}$ is the $n \times (n-r)$ lower right submatrix of $\bm{U}_k$ and $\bm{D}_{k(n-r)}$ is the $(n-r) \times (n-r)$ lower right submatrix of $\bm{D}_k$.

\end{proposition}
\begin{proof}
Let
\begin{equation*}
\begin{aligned}
&\bm{\Lambda}_k  = \begin{bmatrix} \bm{\Lambda}_k^* & \bm{0}_{r ,n-r} \\  \bm{0}_{n-r,r} & \bm{Q}_k \end{bmatrix}, \text{ where }\bm{Q}_k = \frac{1}{2}\sum \limits_{j} \left(\frac{\bm{g}_{kj} (\bm{g}_{kj})^\top}{||\bm{g}_{kj}||_2} + \frac{\bm{h}_{kj} (\bm{h}_{kj})^\top}{||\bm{h}_{kj}||_2}\right), \\
& \bm{q}_k  = \begin{bmatrix} \bm{q}^*_k \\ \bm{0}_{n-r,1} \end{bmatrix}, 
 \xi_k = \xi_k^* + s_k, \text{ where } s_k = \frac{1}{2}\sum \limits_{j} (||\bm{g}_{kj}||_2 + ||\bm{h}_{kj}||_2).
\end{aligned}
\end{equation*}
Then it can be verified that $(\bm{w}^{*},b^*,\bm{\xi},\bm{\Lambda}_k,\bm{q}_k,\bm{v}_{kj}^*,\bm{u}_{kj}^*,t_k^*,r_k^*,\tau_k^*)$ is a set of feasible solutions of the model $M_n$. We only need to test and verify the SDP constraints.

{\small
\begin{equation*}
\begin{aligned}
& \begin{bmatrix}
\bm{\Lambda}_k & \frac{1}{2} \left(\bm{q}_k+ \bm{S}_k^{(r)\top} (y_k \bm{w}^*+\bm{v}^*_{kj})\right) \\[1em]
\frac{1}{2} \left(\bm{q}_k+ \bm{S}_k^{(r)\top} (y_k \bm{w}^*+\bm{v}^*_{kj})\right)^\top &
\begin{aligned} & t_k^* - \tau_k^* + y_k(\bm{w}^{*\top} \bm{\mu}_k + b^*) - 1+\xi_k  \\ & - (\hat{\bm{x}}_{kj} - \bm{\mu}_k)^\top \bm{v}_{kj}^* -\sqrt{\bar{\gamma}_{kj}}||\bm{A}_{kj}^\top \bm{v}^*_{kj}||_q \end{aligned} 
\end{bmatrix}\\[1em]
= & \begin{bmatrix} 
 \bm{\Lambda}_k^* & \bm{0}_{r ,n-r} & \frac{1}{2}\left(\bm{q}_k^* + \bm{S}_k^{(r)\top} (y_k \bm{w}^*+\bm{v}^*_{kj}) \right)\\[1em]  
\bm{0}_{n-r,r} & \bm{Q}_k & \frac{1}{2} \bm{g}_{kj} \\[1em]
\frac{1}{2}\left(\bm{q}_k^* + \bm{S}_k^{(r)\top} (y_k \bm{w}^*+\bm{v}^*_{kj})\right)^\top & \frac{1}{2}  \bm{g}_{kj}^\top & \begin{aligned} & t_k^* - \tau_k^* + y_k(\bm{w}^{*\top} \bm{\mu}_k + b^*) - 1+\xi_k^* + s_k \\ & - (\hat{\bm{x}}_{kj} - \bm{\mu}_k)^T\top\bm{v}_{kj}^* -\sqrt{\bar{\gamma}_{kj}}||\bm{A}_{kj}^T\bm{v}^*_{kj}||_q \end{aligned} 
\end{bmatrix}\\[1em]
= & \begin{bmatrix} 
 \bm{\Lambda}_k^* & \bm{0}_{r,n-r} & \frac{1}{2}\left(\bm{q}_k^* + \bm{S}_k^{(r)\top} (y_k \bm{w}^*+\bm{v}^*_{kj}) \right)\\[1em]  
\bm{0}_{n-r,r} & \bm{0}_{n-r,n-r}  & \bm{0}_{n-r,1} \\[1em]
\frac{1}{2}\left(\bm{q}_k^* + \bm{S}_k^{(r)\top} (y_k \bm{w}^*+\bm{v}^*_{kj}) \right)^\top & \bm{0}_{1,n-r} & \begin{aligned} & t_k^* - \tau_k^* + y_k(\bm{w}^{*\top} \bm{\mu}_k + b^*) - 1+\xi_k^* \\ & - (\hat{\bm{x}}_{kj} - \bm{\mu}_k)^\top\bm{v}_{kj}^* -\sqrt{\bar{\gamma}_{kj}}||\bm{A}_{kj}^\top\bm{v}^*_{kj}||_q \end{aligned} 
\end{bmatrix}\\[1em]
+ & \begin{bmatrix} 
 \bm{0}_{r,r} & \bm{0}_{r ,n-r} & \bm{0}_{r,1}\\[1em]  
\bm{0}_{n-r,r} & \bm{Q}_k & \frac{1}{2} \bm{g}_{kj} \\[1em]
 \bm{0}_{1,r} & \frac{1}{2} \bm{g}_{kj}^\top & s_k
\end{bmatrix} \succeq 0, 
\end{aligned}
\end{equation*}
}
It is the same for other SDP constraints. 
Hence,
\begin{equation*}
\begin{aligned}
v^*(n) - v^*(r) & \leqslant \frac{1}{2}||\bm{w}^*||_2^2 + C(\xi_1 + \xi_2) - (\frac{1}{2}||\bm{w}^*||_2^2 + C(\xi_1^* + \xi_2^*)) \\
& = C(s_1+s_2)  = \frac{C}{2}\sum \limits_{k,j} (||\bm{g}_{kj}||_2 + ||\bm{h}_{kj}||_2).
\end{aligned}
\end{equation*}
\end{proof}

The upper bound of the gap $v^*(n)-v^*(r)$ given in Proposition \ref{prop:app_quality} shows that the approximation quality is controlled by the the smallest $n-r$ eigenvalues of $\bm{\Sigma}_k$.  Therefore, the approximation works well if some eigenvalues of the covariance matrices are extremely small.

\section{Numerical experiments}\label{sec:num}
In this section, we tested the performance of GDRC SVM model (\ref{eq:gdrcc_svm}) using the ambiguity set (\ref{eq:ambiguity}) and its approximation, denoted as GDRC-SVM and GDRC-SVM-app respectively, through several sets of numerical experiments. For comparison, we adopted the deterministic SVM model (\ref{eq:svm}), the DRC SVM model (\ref{eq:drcc_svm}) using the ambiguity set (\ref{eq:ambiguity_known}) and the DRC SVM model (\ref{eq:drcc_svm}) using the ambiguity set with known covariance matrix and mean perturbation $(E_{F_k^i}[\tilde{\bm{x}}_k^i]-\bm{\mu}_k^i)^{\top}(\bm{\Sigma}_k^{i})^{-1}(E_{F_k^i}[\tilde{\bm{x}}_k^i]-\bm{\mu}_k^i) \leqslant (\nu_k^i)^2$, denoted as SVM, DRC-SVM, DRC-mu-SVM, respectively. 

In the GDRC model, $\bm{\mu}_k$ and $\bm{\Sigma}_k$ are the sample mean and covariance matrix of the sample points in the $k$th class, $Y_{kj}$ and $\phi_{kj}$ are taken as (\ref{eq:setting}) with $p_{kj}=2$. For simplicity, let $m_1 = m_2 = 1$ and omit $j$. Let $\hat{\bm{x}}_1 = (1-\lambda)\bm{\mu}_1 + \lambda \bm{\mu}_2$, $\hat{\bm{x}}_2 =  \lambda \bm{\mu}_1 + (1-\lambda)\bm{\mu}_2$ ($0 \leqslant \lambda < 0.5$), $\bm{A}_k = \bm{\Sigma}_k^{1/2}$. The parameter $\lambda$ and radius $\bar{\gamma}_{k}$ are determined case by case, which can be determined by a cross validation procedure. In the DRC model, we follow some settings from \cite{svm:drcc:Wang}. Specifically, $\bm{\mu}_k^i$ is the value of the sample point and $\bm{\Sigma}_k^i = 0.01 \bm{\Sigma}_k$, $(\nu_k^i)^2 = \frac{n(N_0-1)}{N_0(N_0-n)}F^{-1}_{n,N_0-n}(0.9)$ with $N_0=100$, where $F^{-1}_{n,N_0-n}(0.9)$ is 0.9 quantile of the probability distribution $F(n,N_0-n)$. The common parameters are taken as $C=16$, $\varepsilon = 0.05$.

To quantify the performance of the classification hyperplane, we use classification accuracy defined by
\begin{equation}
acc = \frac{\sum_i \mathds{1}_{y_i^{pr} = y_i}}{N_{test}},
\end{equation}
where $N_{test}$ is the sample size of the testing dataset, $y_i^{pr}$ and $y_i$ are the predicted label and true label of $i$th testing data, respectively.

In this section, all numerical experiments were conducted using MATLAB (R2018b) software. Unless otherwise specified, the models in numerical experiments were solved by CVX \cite{cvx}.

\subsection{Simulated data}
\subsubsection{Impact of core sets}
 
In this subsection, two-dimensional data were generated, to visually show the performance of the GDRC SVM model and the impact of core sets. For $\pm 1$ class, the sample data were generated by normal distribution $N(\bm{\mu}_{\pm},\bm{I}_2)$ with $\bm{\mu}_{\pm} = \pm [1,1]^\top$. For each class, 50 sample points were generated, 10 of which were randomly selected as training data and the rest were used as testing data. 

Figure \ref{fig:coreset_center} shows the classification hyperplane and accuracy given by GDRC SVM model without core-set constraints and with different core sets. The blue and red circles are the points of +1 class and -1 class, respectively. The ones filled inside are training data and the rest are testing data. The black line is the classification hyperplane $\bm{w}^\top \bm{x}+b = 0$, and the blue and red dash lines are the margin lines $\bm{w}^\top \bm{x}+b = \pm 1$. The blue and red ellipsoids are the core sets for +1 class and -1 class, with the small solid stars as the centers, respectively. 

\begin{figure}[htbp]
\centering
\begin{tabular}{cc}
\begin{minipage}{0.45\textwidth}
\includegraphics[width=\textwidth]{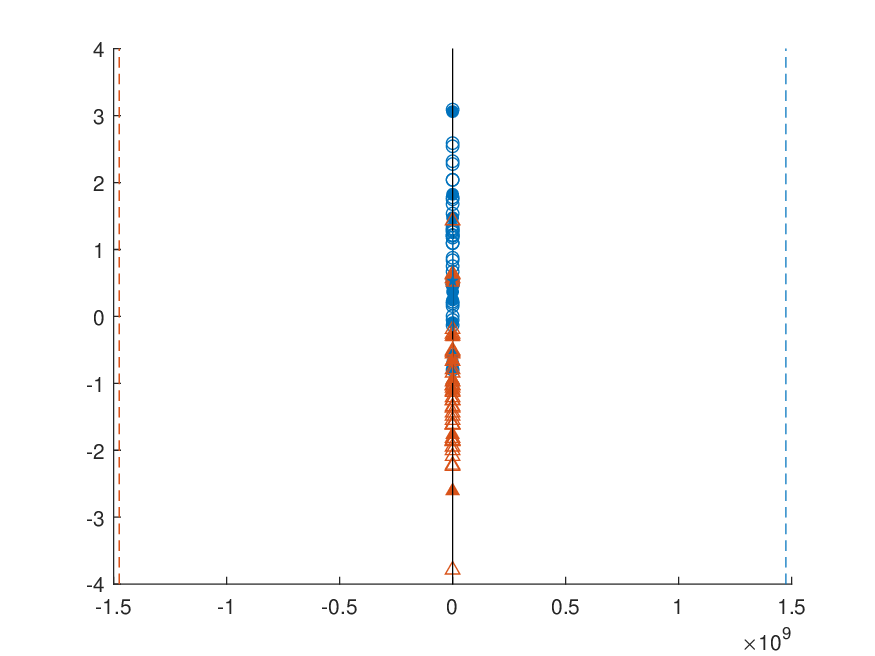}
\end{minipage} &
\begin{minipage}{0.45\textwidth}
\includegraphics[width=\textwidth]{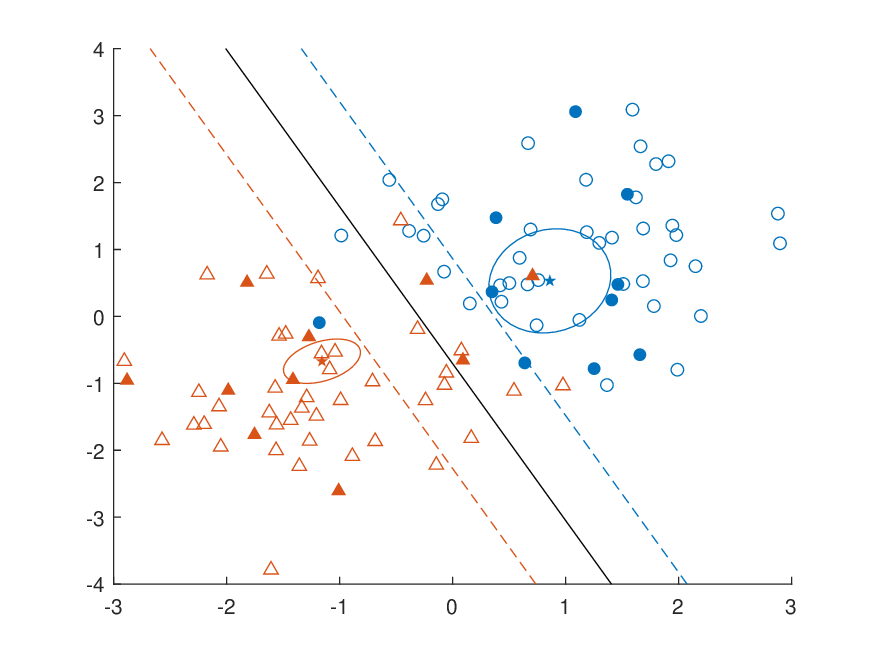}
\end{minipage}  \\
\scriptsize{\makecell{$\theta=0$. The robust region is too large that \\ the resulting $\bm{w}$ will be approximately 0.}} & \scriptsize{$\lambda=0$, acc=$93.75\%$} \\
\begin{minipage}{0.45\textwidth}
\includegraphics[width=\textwidth]{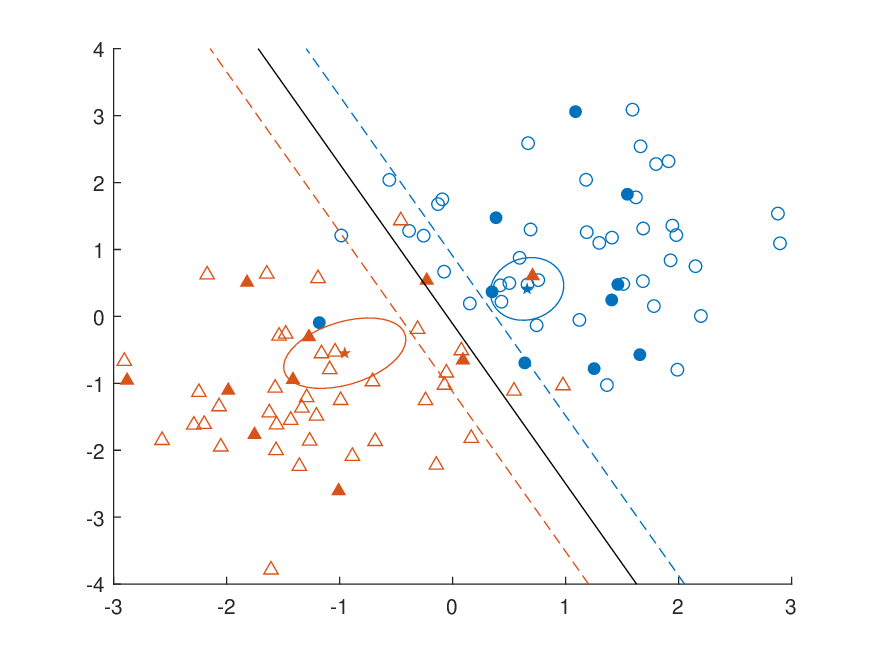}
\end{minipage} &
\begin{minipage}{0.45\textwidth}
\includegraphics[width=\textwidth]{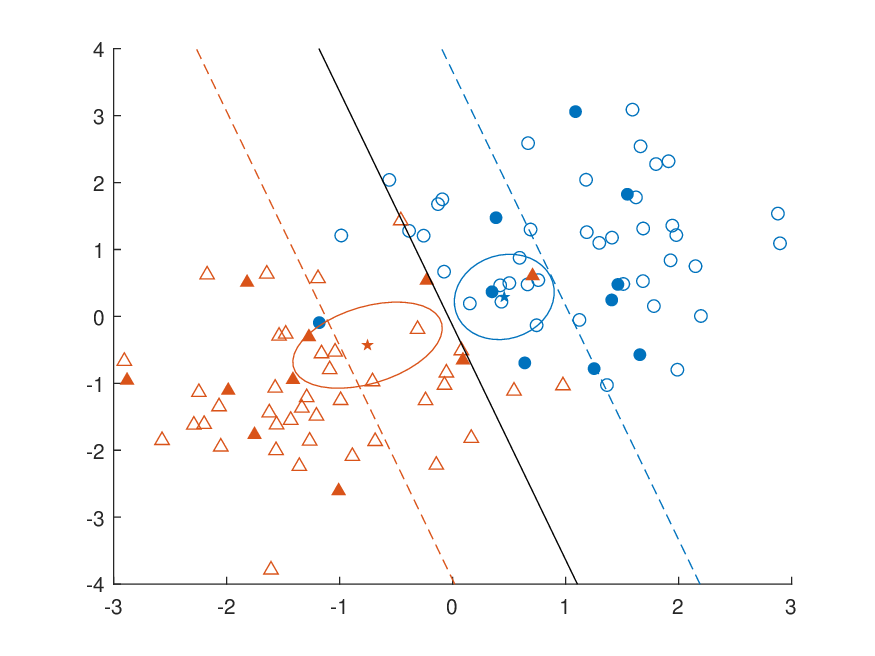}
\end{minipage} \\
\scriptsize{$\lambda=0.1$, acc=$95\%$} & \scriptsize{$\lambda=0.2$, acc=$96.25\%$}\\
\end{tabular}
\caption{Classification with different core sets}\label{fig:coreset_center}
\end{figure}

Given $\theta =0$, that means the core-set constraint vanishes in the ambiguity set of the GDRC SVM model, the optimal solution $\bm{w}^*$ becomes approximately 0 because the robust region is too large. It shows that the core-set constraint is essential and has an important impact on reducing the conservatism of the GDRC SVM model. Given $\theta = 400$, $\lambda=0,0.1,0.2$, the radius of the core set $\bar{\gamma}_{k}$ is given by the $\gamma$ such that $10\%$ training data and the sample mean are contained in $Y_k$. The larger $\lambda$ is, the closer the core set is to the classification hyperplane and the higher the accuracy is. It indicates that it is necessary and meaningful to pay more attention to the sample points near the classification hyperplane. The core sets help capture the important data and improve the classification performance.

If the center of the core set is fixed, the sizes of two core sets have an impact on the classification hyperplane. For the case $\lambda=0$ in Figure \ref{fig:coreset_center}, fix $\bar{\gamma}_1$ and enlarge $\bar{\gamma}_2$ to 2, 4 and 8 times of the original value respectively. Figure \ref{fig:coreset_radius} shows that the classification hyperplane gets closer and closer to the core set of class +1 (the blue ellipsoid), as the radius of the core set of class -1 becomes larger. Especially, though the classification hyperplanes given by ratio=8 and ratio=4 attain the same accuracy, their accuracies for each class are different. As the radius of the core set of class -1 becomes larger, the classification hyperplane attains higher accuracy for class -1. Therefore, it may help maintain fairness by choosing similar radii for two core sets when the centers of the core sets are determined with the same $\lambda$.

\begin{figure}[htbp]
\centering
\includegraphics[width=0.7\textwidth]{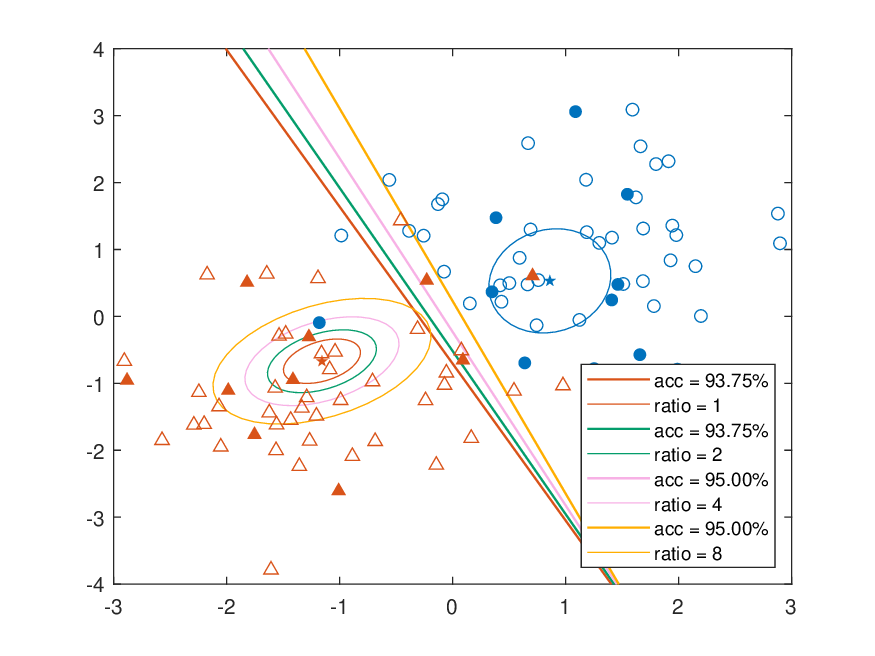}
\caption{Classification result with different $\bar{\gamma}_2$}\label{fig:coreset_radius}
\end{figure}

\subsubsection{Impact of sample size and dimension}

In this subsection, $n$-dimensional data were generated, to show the impact of the sample size and dimension and compare the performance of GDRC-SVM with that of SVM, DRC-SVM and DRC-mu-SVM. 
The sample data of $\pm 1$ class were generated by normal distribution $N(\bm{\mu}_{\pm},5\bm{I}_n)$ with $\bm{\mu}_{\pm} = \pm \mathbf{1}_n$, respectively. To test the impact of the sample size, fixed $n=30$ and $N=600$, for each class, $N/2$ sample points were generated, $20\%, 40\%, 60\%$ and $80\%$ of which were randomly selected as training data respectively. To test the impact of the data dimension, given $n=10,20,30,40,50$ respectively and fixed $N=600$, for each class, $N/2$ sample points were generated, $20\%$ of which were randomly selected as training data. The experiments were repeated 20 times. The mean and standard deviation of classification accuracy and calculation time are shown in Table \ref{table:samp_size} and \ref{table:dim}.

\begin{table}[tbhp]
\centering
\caption{Impact of sample size}\label{table:samp_size}
\setlength{\tabcolsep}{3mm}{
\resizebox{\textwidth}{22mm}{
\begin{tabular}{ccccccccccccc}
\toprule[1.5pt]
\multirow{2}{*}{$N_{train}$} &~& \multicolumn{2}{c}{SVM}& ~ & \multicolumn{2}{c}{DRC-SVM} & ~& \multicolumn{2}{c}{DRC-mu-SVM}& ~ & \multicolumn{2}{c}{GDRC-SVM} \\\cline{3-4} \cline{6-7} \cline{9-10} \cline{12-13} 
~ &~& acc($\%$) & time(s)& ~ & acc($\%$) & time(s)& ~ & acc($\%$) & time(s) & ~& acc($\%$) & time(s) \\
\midrule[1pt]
$20\%N$ & mean & 97.82 & 0.96  & ~ & 97.60  & 3.40 & ~ & 97.63 & 3.38 & ~ & \textbf{98.43} & 5.78 \\ 
~ & std & 0.69 & 0.06  & ~ & 0.57  & 0.14 & ~ & 0.57 & 0.14 & ~ & 0.48 & 1.05  \\ 
$40\%N$ & mean & 98.60 & 1.45  & ~ & 98.58  & 7.23 & ~ & 98.49 & 7.21 & ~ & \textbf{98.99} & 4.05 \\ 
~ & std & 0.46 & 0.10  & ~ & 0.48  & 0.06 & ~ & 0.57 & 0.06 & ~ & 0.40 & 0.73  \\ 
$60\%N$ & mean & 97.79 & 2.08  & ~ & 97.73  & 11.83 & ~ & 97.77 & 11.75 & ~ & \textbf{99.13} & 4.12 \\ 
~ & std & 1.13 & 0.04   & ~ & 1.13  & 0.10 & ~ & 1.13 & 0.07 & ~ & 0.49 &  0.70 \\ 
$80\%N$ & mean & 98.63 & 3.18  & ~ & 98.38  & 17.41 & ~ & 98.29 & 17.28 & ~ & \textbf{99.25} & 3.99  \\ 
~ & std & 1.09 & 0.09  & ~ & 1.28  & 0.22 & ~ & 1.28 & 0.22 & ~ & 0.66 & 0.48  \\ 
\bottomrule[1.5pt]
\end{tabular} 
}}
\end{table} 

\begin{table}[tbhp]
\centering
\caption{Impact of data dimension}\label{table:dim}
\setlength{\tabcolsep}{3mm}{
\resizebox{\textwidth}{26mm}{
\begin{tabular}{ccccccccccccc}
\toprule[1.5pt]
\multirow{2}{*}{Dimension} &~& \multicolumn{2}{c}{SVM}& ~ & \multicolumn{2}{c}{DRC-SVM} & ~& \multicolumn{2}{c}{DRC-mu-SVM}& ~ & \multicolumn{2}{c}{GDRC-SVM} \\\cline{3-4} \cline{6-7} \cline{9-10} \cline{12-13} 
~ &~& acc($\%$) & time(s)& ~ & acc($\%$) & time(s)& ~ & acc($\%$) & time(s) & ~& acc($\%$) & time(s) \\
\midrule[1pt]
$n=10$ & mean & 90.46 & 0.99  & ~ & 91.74  & 3.04 & ~ & 91.61 & 2.98 & ~ & \textbf{92.20} & 0.98 \\ 
~ & std & 1.62 & 0.28  & ~ & 0.79  & 0.17 & ~ & 0.79 & 0.15 & ~ & 0.75 & 0.19  \\ 
$n=20$ & mean & 94.30 & 1.02  & ~ & 93.82  & 3.23 & ~ & 94.33 & 3.22 & ~ & \textbf{96.75} & 1.64 \\ 
~ & std & 1.63 & 0.15  & ~ & 1.50  & 0.23 & ~ & 1.31 & 0.28 & ~ & 0.59 & 0.19  \\ 
$n=30$ & mean & 96.81 & 1.02  & ~ & 96.44  & 3.48 & ~ & 96.38 & 3.58 & ~ & \textbf{97.56} & 6.09 \\ 
~ & std & 0.95 & 0.09  & ~ & 0.92  & 0.16 & ~ & 0.93 & 0.33 & ~ & 0.83 &  0.55 \\ 
$n=40$ & mean & 99.05 & 0.94  & ~ & 98.80  & 4.16 & ~ & 98.76 & 4.13 & ~ & \textbf{99.65} & 10.98 \\ 
~ & std & 0.44 &  0.06 & ~ & 0.50  & 0.15 & ~ & 0.48 & 0.12 & ~ & 0.20 & 1.74  \\ 
$n=50$ & mean & 99.61 & 0.95  & ~ &  99.54 & 5.05 & ~ & 99.51 & 5.06  & ~ & \textbf{99.74} & 31.81 \\ 
~ & std & 0.23 & 0.03  & ~ & 0.31  & 0.16 & ~ & 0.30 & 0.14 & ~ & 0.21 & 2.13  \\ 
\bottomrule[1.5pt]
\end{tabular} 
}}
\end{table}

The bold values indicate the best performance of GDRC-SVM. Actually GDRC-SVM not only attains the highest average classification accuracy, but also provides the smallest standard deviation, which demonstrates the robustness of GDRC-SVM. As the size of training sample increases, the calculation time of SVM, DRC-SVM and DRC-mu-SVM obviously increases, whereas the calculation time of GDRC-SVM is almost unaffected by the sample size. Therefore, the GDRC SVM model is a great choice when the sample size is large. However, the calculation time of GDRC-SVM rapidly increases as the data dimension increases. Hence, the approximation of high-dimensional GDRC-SVM is needed to reduce the calculation time.

\subsubsection{Performance of approximation}
In this subsection, $n$-dimensional data were generated, to test the performance of the approximation of GDRC-SVM proposed in Section \ref{sec:app}. 
The sample data of $\pm 1$ class were generated by normal distribution $N(\bm{\mu}_{\pm},10\bm{I}_n)$ with $\bm{\mu}_{\pm} = \pm \mathbf{1}_n$, respectively. Fixed $n=50$ and $N$=1 000, for each class, $N/2$ sample points were generated, $20\%$ of which were randomly selected as training data. To test the approximation quality, we considered the approximations of $r = 100\%n, 50\%n, 30\%n, 10\%n$, respectively. 
The experiments were repeated 20 times. The results are shown in Table \ref{table:app_s}, where ``val" means the optimal value of the model, ``gap" means the true gap $v^*(n)-v^*(r)$, and ``ub" means the theoretical upper bound (\ref{eq:gap}) given in Proposition \ref{prop:app_quality}.

\begin{table}[tbhp]
\centering
\caption{Performance of approximation model}\label{table:app_s}
\setlength{\tabcolsep}{2mm}{
\resizebox{\textwidth}{22mm}{
\begin{tabular}{ccccccccc}
\toprule[1.5pt]
\multirow{2}{*}{Model} & \multicolumn{2}{c}{acc($\%$)} &~&  \multicolumn{2}{c}{time(s)} &  \multirow{2}{*}{val} & \multirow{2}{*}{gap} & \multirow{2}{*}{ub} \\ \cline{2-3} \cline{5-6}
~ & mean & std &~&  mean & std & ~ & ~ & ~ \\
\midrule[1pt]
GDRC-SVM & 98.58 & 0.28 &~& 34.50 & 4.66 & 0.0356 & - & - \\
GDRC-SVM-app(100$\%$) & 98.58 & 0.28 &~& 13.69 & 0.89 & 0.0357 & - & - \\
GDRC-SVM-app(50$\%$) & 98.49 & 0.29 &~& 4.42 & 0.21 & 0.0158 & 0.0199  & 15.0247 \\
GDRC-SVM-app(30$\%$) & 98.51 & 0.34 &~& 2.11 & 0.11 & 0.0128 & 0.0228 & 15.2336 \\
GDRC-SVM-app(10$\%$) & 98.57 & 0.40 &~& 0.96 & 0.04 & 0.0108 & 0.0249 & 12.2731 \\
\bottomrule[1.5pt]
\end{tabular} 
}}
\end{table}

The calculation time of GDRC-SVM-app significantly decreases as the approximation dimension $r$ decreases. The calculation time of GDRC-SVM-app(100$\%$) is less than that of GDRC-SVM, probably because there are fewer matrix multiplications when solving the model (\ref{eq:svm_gdrcc_norm_app}). 
Additionally, when $n=r$, the optimal values of GDRC-SVM-app and GDRC-SVM should be equal. The small deviation may be caused by errors in the calculation. 
The theoretical upper bound given by Proposition \ref{prop:app_quality} is much higher than the true gap. The theoretical upper bound of the smallest $r$ dimensional approximation is smallest, which also indicates that this upper bound is not a good estimation. 
Thus, the theoretical upper bound in Proposition \ref{prop:app_quality} may potentially be improved. 
The SVM combined with PCA may improve the generalization ability, but it depends on the characteristics of the dataset. Meanwhile, some important information may be lost after the dimension reduction via PCA. In this set of experiments, the approximation model with fewer principal components attains higher accuracy, whereas the robustness of the approximation model decreases as the number of principal components decreases.

\subsection{Real dataset}
Besides the simulated data, we also tested GDRC-SVM using some well known datasets, including Wisconsin breast cancer, Balance-scale and Ionosphere from UCI databases\cite{uci}, and Mushrooms and Cod-RNA from LIBSVM\cite{libsvm}. The summary of these datasets is listed in Table \ref{table:datasets}. In this subsection, we used $20\%$ of data points for training and the rest for testing. Each group of numerical experiments was repeated 20 times and the mean and standard deviation of classification accuracy and calculation time were recorded.

The datasets Wisconsin breast cancer and Balance-scale were tested to compare the performance of SVM, DRC-SVM, DRC-mu-SVM and GDRC-SVM. The mean and standard deviation of classification accuracy and calculation time are shown in Table \ref{table:real_res}. 
The numerical experiments on the datasets Ionosphere and Mushrooms mainly show the performance of the approximation of GDRC SVM model. When using the Mushrooms dataset, the models DRC-SVM, DRC-mu-SVM and GDRC-SVM were not tested due to too long calculation time required for the high-dimensional model. The results are shown in Table \ref{table:app_r}. 
The sample size of the dataset Cod-RNA is very large. Hence, we only compared the model SVM and GDRC-SVM on the dataset Cod-RNA. Especially, SVM was solved by LIBSVM in this group of experiments, because solving the model (\ref{eq:svm}) via CVX required much more time. The results are shown in Table \ref{table:rna}.
\begin{table}[tbhp]
\centering
\caption{Summary of the datasets}\label{table:datasets}
\setlength{\tabcolsep}{2mm}{
\resizebox{\textwidth}{12mm}{
\begin{tabular}{ccccccc}
\toprule[1.5pt]
Data set & ~ & Wisconsin breast cancer & Balance-scale & Ionosphere & Mushrooms & Cod-RNA \\
\midrule[1pt]
Dimension & $n$ & 9 & 4 & 34  & 112 & 8 \\[0.5em]
\multirow{2}{*}{Sample size} & $N_+$ & 239 & 288 & 225 & 3916 & 19845 \\
~ & $N_-$ & 444 & 288 & 126 & 4208 & 39690 \\
\bottomrule[1.5pt]
\end{tabular} 
}}
\end{table}

\begin{table}[tbhp]
\centering
\caption{Numerical results on the real data sets}\label{table:real_res}
\setlength{\tabcolsep}{2mm}{
\resizebox{\textwidth}{22mm}{
\begin{tabular}{ccccccccccccc}
\toprule[1.5pt]
\multirow{2}{*}{Model} & ~ & \multicolumn{5}{c}{Wisconsin breast cancer} & ~ & \multicolumn{5}{c}{Balance-scale} \\ \cline{3-7} \cline{9-13} 
\addlinespace[2pt]
~ & ~ & \multicolumn{2}{c}{acc($\%$)}  & ~ & \multicolumn{2}{c}{time(s)}&  ~ & \multicolumn{2}{c}{acc($\%$)} & ~ & \multicolumn{2}{c}{time(s)} \\\cline{3-4} \cline{6-7} \cline{9-10} \cline{12-13} 
\addlinespace[2pt]
~ & ~ & mean & std  & ~ & mean & std & ~ & mean & std  & ~ & mean & std \\
\midrule[1pt]
SVM & ~ & 94.24 & 1.84  & ~ &  0.90  & 0.04 & ~ & 94.78 & 0.75  & ~ & 0.80  & 0.03  \\
DRC-SVM & ~ & 95.40  & 1.35   & ~ & 3.08 & 0.04 & ~ & 95.10 & 0.88  & ~ & 2.61 & 0.03 \\ 
DRC-mu-SVM & ~ & 95.69 & 1.23  & ~ & 3.08 & 0.03 & ~ & 94.96 & 0.87  & ~ & 2.60 & 0.03 \\
GDRC-SVM & ~ & \textbf{95.79} & 0.79  & ~ & 0.81 & 0.03 & ~ & \textbf{95.15} & 0.68  & ~ & 0.81 & 0.05 \\ 
\bottomrule[1.5pt]
\end{tabular} 
}}
\end{table}

\begin{table}[tbhp]
\centering
\caption{Numerical results with approximation model }\label{table:app_r}
\setlength{\tabcolsep}{2mm}{
\resizebox{\textwidth}{25mm}{
\begin{tabular}{ccccccccccccc}
\toprule[1.5pt]
\multirow{2}{*}{Model} & ~ & \multicolumn{5}{c}{Ionosphere} & ~ & \multicolumn{5}{c}{Mushrooms} \\ \cline{3-7} \cline{9-13} 
\addlinespace[2pt]
~ & ~ & \multicolumn{2}{c}{acc($\%$)}  & ~ & \multicolumn{2}{c}{time(s)}&  ~ & \multicolumn{2}{c}{acc($\%$)} & ~ & \multicolumn{2}{c}{time(s)} \\\cline{3-4} \cline{6-7} \cline{9-10} \cline{12-13} 
\addlinespace[2pt]
~ & ~ & mean & std  & ~ & mean & std & ~ & mean & std  & ~ & mean & std \\
\midrule[1pt]
SVM & ~ & 83.20 & 2.13  & ~ & 0.74 & 0.06 & ~ & \textbf{99.93} & 0.09  & ~ & 26.42 & 1.86 \\ 
DRC-SVM & ~ & 82.33 & 2.96  & ~ & 2.28 & 0.15 & ~ & - &  - & ~ & - & - \\ 
DRC-mu-SVM & ~ & 82.44 & 3.13  & ~ & 2.26 & 0.09 & ~ & - &  - & ~ & - & - \\ 
GDRC-SVM & ~ & 84.96 & 1.09  & ~ & 8.04 & 1.15 & ~ & - & -  & ~ & - & - \\ 
GDRC-SVM-app(50$\%$) & ~ & \textbf{85.30} & 0.95  & ~ & 0.95 & 0.07 & ~ & 99.88 & 0.13  & ~ & 39.87 & 2.11 \\ 
GDRC-SVM-app(20$\%$) & ~ & 85.11 & 1.63  & ~ & 0.67 & 0.04 & ~ & 99.23 & 0.30  & ~ & 1.74 & 0.09 \\ 
\bottomrule[1.5pt]
\end{tabular} 
}}
\end{table} 

\begin{table}[tbhp]
\centering
\caption{Numerical results on Cod-RNA dataset}\label{table:rna}
\begin{tabular}{ccccccc}
\toprule[1.5pt]
\multirow{2}{*}{Model} & ~ & \multicolumn{2}{c}{acc($\%$)}  & ~ & \multicolumn{2}{c}{time(s)}\\\cline{3-4} \cline{6-7} 
\addlinespace[2pt]
~ & ~ & mean & std  & ~ & mean & std \\
\midrule[1pt]
SVM & ~ & 92.90 & 0.20 & ~ & 23.92  & 4.76  \\
GDRC-SVM & ~ & \textbf{93.46} & 0.10 & ~ & 0.93 & 0.08  \\  
\bottomrule[1.5pt]
\end{tabular} 
\end{table} 

The accuracy results shown in Table \ref{table:real_res}, \ref{table:app_r} and \ref{table:rna} demonstrate that GDRC SVM model can achieve more accurate classifications and stronger robustness than other tested SVM models on most of datasets. From the perspective of calculation time, when the dataset dimension is low, GDRC-SVM outperforms other tested SVM models, especially when the dataset has large sample size; when the dataset dimension is high, the approximation of GDRC SVM model effectively reduces the calculation time.

Furthermore, the numerical results on the dataset Ionosphere show that the approximation based on PCA can achieve higher accuracy and greater robustness than the original GDRC SVM model if the number of principal components is well determined. 
For the datasets Ionosphere and Mushrooms, the approximation with $50\%$ principal components performs better than the approximation with $20\%$ principal components, demonstrating that too few principal components may reduce the classification accuracy and robustness of the model. However, sacrificing a little classification accuracy can shorten the calculation time, especially when the data dimension is high. Hence, there is a trade-off between the classification quality and calculation time when using the approximation model of GDRC SVM.

Combining the results shown in Table \ref{table:real_res} and \ref{table:rna}, the GDRC SVM model takes less time to find the optimal classification hyperplane when the data dimension is small, regardless of the sample size. The fact that the calculation time does not depend on the sample size makes the GDRC SVM model more suitable for massive datasets. 

Overall, the numerical results on the datasets listed in Table \ref{table:datasets} illustrate the great performance of GDRC SVM and its approximation and indicate the potential value of GDRC SVM on various datasets.

\section{Conclusion}\label{sec:cons}
This study introduced a globalized distributionally robust chance-constrained SVM model, where the globalization means considering the uncertainties in the sample population for each class of data in modelling. In GDRC SVM model, the core sets are constructed to capture some regions near the potential classification hyperplane, so that the core-set constraint is not only helpful to reduce the conservatism but also to improve the classification accuracy. We obtain the equivalent SDP reformulation for the GDRC SVM model when the core sets and distance functions are properly chosen. In order to efficiently handle the high-dimensional dataset, the approximation based on the principal component analysis is given, which may also help improve the classification accuracy with appropriate choice of principal components. Numerical experiments show the great performance of GDRC SVM model and its approximation from the perspective of classification accuracy, the robustness of the model and the calculation time.

\clearpage
\bibliographystyle{plain}
\bibliography{ref}

\end{document}